\documentclass[11pt]{amsart}
\usepackage[letterpaper,margin=1.2in]{geometry}
\usepackage{amsbsy,amsfonts,amsmath,amssymb,amscd,amsthm,mathrsfs}
\usepackage{subfigure,enumitem,graphicx,epsfig,color}
\usepackage[colorlinks=true,linkcolor=blue,citecolor=green]{hyperref}
\usepackage{amsmath}
\usepackage[utf8]{inputenc}
\usepackage[T1]{fontenc}
\usepackage[version=4]{mhchem}
\usepackage{stmaryrd}
\usepackage{bbold}
\usepackage{graphicx}
\usepackage[export]{adjustbox}
\normalsize

\newtheorem{theorem}{Theorem}[section]
\newtheorem{lemma}[theorem]{Lemma}

\newtheorem{definition}[theorem]{Definition}
\newtheorem{corollary}[theorem]{Corollary}
\newtheorem{remark}[theorem]{Remark}

\newtheorem{proposition}[theorem]{Proposition}

\newtheorem*{theorem A}{Theorem A}
\newtheorem*{theorem B}{Theorem B}
\theoremstyle{definition}
\newtheorem{example}[theorem]{Example}

\DeclareMathOperator{\supp}{supp}

\DeclareMathOperator{\mdim}{mdim}
\DeclareMathOperator{\diam}{diam}

\title[On the mean $\Psi$-intermediate dimensions]{On the mean $\Psi$-intermediate dimensions}
\author[Yu Liu, Bilel Selmi \& Zhiming Li]{Yu Liu, Bilel Selmi \& Zhiming Li}
\address{{\bf Yu Liu and Zhiming Li: } School of Mathematics,
Northwest University,
Xi'an, 710127,
P.R.China }
\email{liuyu\_yzr@126.com;  china-lizhiming@163.com}

\address{{\bf Bilel Selmi: }  Analysis, Probability \& Fractals Laboratory LR18ES17, Department of Mathematics, Faculty of Sciences of Monastir,
University of Monastir, 5000-Monastir, Tunisia }
\email{bilel.selmi@fsm.rnu.tn, bilel.selmi@isetgb.rnu.tn}

\begin{document}
\maketitle
\thispagestyle{empty}
\begin{abstract}
In this paper, we introduce the mean $\Psi$-intermediate dimension which has a value between the mean Hausdorff dimension and the metric mean dimension, and prove the equivalent definition of the mean Hausdorff dimension and the metric mean dimension. Furthermore, we delve into the core properties of the mean $\Psi$-intermediate dimensions. Additionally, we establish the mass distribution principle, a Frostman-type lemma, Hölder distortion, and derive the corresponding product formula. Finally, we provide illustrative examples of the mean $\Psi$-intermediate dimension, demonstrating its practical applications.
\end{abstract}
\section{Introduction}

To comprehend how fractal subsets occupy space at small scales within a metric space, it is imperative to investigate diverse dimensions. Two particularly notable dimensions are the Hausdorff dimensions and the box dimensions. For numerous natural sets, these dimensions diverge, indicating a certain degree of inhomogeneity. The fundamental difference lies in the requirement for covering sets: the box dimension mandates them to be of uniform size, whereas the Hausdorff dimension does not impose such a constraint. Falconer et al. proposed a family of dimensions, termed intermediate dimensions, parameterized by $\theta \in [0,1]$, requiring that the covering sets have sizes within the intervals of $\left[\varepsilon^{1 / \theta}, \varepsilon\right]$. And Falconer et al. unites various properties of intermediate dimensions and exemplifies them through instances, see \cite{12,13}. The Hausdorff and box dimensions correspond to the extreme cases of $\theta = 0$ and $\theta = 1$, respectively. 
Intermediate dimensions have been explored in various studies, such as \cite{BA, 4,7,9, Douzi,11,n2,n3,36}, with explicit computations for fractal sets like Bedford-McMullen carpets \cite{3} and infinitely generated self-conformal sets \cite{2}.
For a given set, intermediate dimensions are continuous for each $\theta \in (0,1]$. For many sets, including Bedford-McMullen carpets (see \cite[Section 4]{13}, \cite{3}) and polynomial sequences (see \cite[Proposition 3.1]{13}), they are also continuous at $\theta = 0$, fully interpolating between the Hausdorff and box dimensions. This continuity at $\theta = 0$ is significant for the box dimensions of set projections \cite{6} and the set's images under stochastic processes like fractional Brownian motion \cite{5}. However, intermediate dimensions are discontinuous for many sets at $\theta = 0$ or remain constant at the box dimension value, offering limited insight into the set.

Intermediate dimensions represent a form of dimension interpolation, a relatively recent concept that has garnered significant interest. For an overview, see \cite{17}. This concept creates a family of dimensions that lie between two different notions, retaining characteristics of both while offering additional insights about the sets. This approach not only provides an interesting perspective, but also helps explain why the two endpoint dimensions can yield different values for certain sets.
The exploration of the dimension theory of function graphs, including the ``popcorn'' function and its higher-dimensional analogues, is presented by \cite{1}. The authors calculate various dimensions using the Chung–Erdős inequality, Duffin–Schaeffer estimates, and Euler's totient bound. The applications of their work involve bounding the dimensions of fractional Brownian images and studying Hölder distortion.
Another example of dimension interpolation is the Assouad spectrum, introduced by Fraser and Yu in \cite{20}. The Assouad spectrum spans the range between the upper box dimension and the Assouad dimension, where the Assouad dimension is introduced in \cite{15} ,providing insights into the "thickest" part of a set.
A more generalized class of dimensions, introduced in \cite{20}, was further developed by García, Hare, and Mendivil in \cite{23}, and studied extensively in \cite{n1,n4,16,19,22,24,25,37}. These dimensions allow for more flexible relative scales compared to the Assouad spectrum, offering detailed geometric information about sets with a quasi-Assouad dimension lower than the Assouad dimension. These Assouad-like dimensions partly inspired our consideration of the $\Psi$-intermediate dimensions.

The concept of mean dimension, introduced by Gromov in 1999 \cite{gro}, has emerged as a pivotal topological invariant in the realm of dynamical systems. This measure has since proven to be indispensable in addressing embedding problems within this field \cite{L, LT}. Expanding upon Gromov's framework, Lindenstrauss and Weiss introduced the metric mean dimension, establishing it as an upper bound for mean dimension and showcasing its efficacy in characterizing the topological intricacies of systems with infinite entropy. 

In 2013, Li introduced mean dimensions for continuous actions of countable sofic groups on compact metrizable spaces in \cite{LHF1}. These dimensions generalize the Gromov-Lindenstrauss-Weiss mean dimensions for actions of countable amenable groups and prove useful in distinguishing continuous actions of countable sofic groups with infinite entropy.
And in 2018, Li and Liang introduced an invariant, termed mean rank, for any module $\mathscr{M}$ belonging to the integral group ring of a discrete amenable group $\Gamma$ in \cite{LHF2}. This invariant is analogous to the rank of an abelian group. They demonstrate that the mean dimension of the induced $\Gamma$-action on the Pontryagin dual of $\mathscr{M}$, the mean rank of $\mathscr{M}$, and the von Neumann-Lück rank of $\mathscr{M}$ all coincide.

This development parallels the classical variational principle \cite{Wal82}, which bridges ergodic theory and topological dynamics. A key frontier lies in integrating ergodic theoretic concepts into mean-dimension theory, thereby establishing novel variational principles for the metric mean dimension. In particular, Lindenstrauss and Tsukamoto's seminal work in 2018 \cite{LT} unveiled a groundbreaking variational principle for the metric mean dimension, using rate-distortion functions from information theory. The resilience of this principle is evident through its compatibility with various candidates of the rate-distortion function \cite{GS21, Shi22, VV17, Wu21}. The allure and relevance of mean dimension have been underscored by a burgeoning interest in recent years, with an array of references delving into its applications, particularly within the context of the embedding problem in dynamical systems \cite{JMA1, JMA2, Backes, Cheng, Liu, LY, Yang}.

In this paper, we introduce the mean $\Psi$-intermediate dimensions, which allow us to cover set sizes within broader intervals defined by $[\Psi(\varepsilon), \varepsilon]$ for more general functions $\Psi$. These dimensions provide more detailed geometric information than mean dimensions, particularly for sets where the latter are discontinuous at $\theta = 0$. One of our most significant findings is that for any compact set, there exists a family of functions $\Psi$ that can fully interpolate between the mean Hausdorff and box dimensions.

Next, let us dive into the content of this paper. In Section \ref{Section2}, we define the mean Hausdorff dimension and the metric mean dimension and clarify the concepts of the upper and lower mean $\Psi$-intermediate dimensions. In Section \ref{Section3}, we establish equivalent definitions for the Hausdorff mean dimension and the metric mean dimension, extending this equivalence to the upper and lower mean $\Psi$-intermediate dimensions as well. Section \ref{Section4} investigates the fundamental properties of the mean $\Psi$-intermediate dimension, providing proofs for the mass distribution principle, Frostman-type lemma, and Hölder distortion, and deriving the associated product formula. Finally, Section \ref{Section5} presents three illustrative examples of the mean $\Psi$-intermediate dimension, demonstrating practical applications, and concluding our discussion.

\section{Preliminaries}\label{Section2}
Let $(\Omega,\Gamma)$ be a topological dynamical system (TDS), where $(\Omega,d)$ is a metric space and $\Gamma:\Omega\to \Omega$ is continues. Let $\mathcal{P}(\Omega)$ denote the power set of $\Omega$ and $\mathbb{N}=\left\{1,2,\dots\right\}$.
For $N\in\mathbb{N}$ and $\widetilde{\mathbf{u}},\widetilde{\mathbf{v}} \in \Omega$, set
$$
 d_{N}(\widetilde{\mathbf{u}},\widetilde{\mathbf{v}}) = \max_{0 \leq i \leq N-1} d\left(\Gamma^{i}(\widetilde{\mathbf{u}}), \Gamma^{i}(\widetilde{\mathbf{v}})\right).
$$
For any $\widetilde{\mathbf{u}}\in \Omega$ and $r\textgreater 0$ , let
$$\mathscr{B}_{d}(\widetilde{\mathbf{u}}, r)=\left\{\widetilde{\mathbf{v}} \in \Omega \big|\;\; \mathrm{d}(\widetilde{\mathbf{u}}, \widetilde{\mathbf{v}})\leq r\right\}\quad\text{and}\quad {U}_{d}(\widetilde{\mathbf{u}}, r)=\left\{\widetilde{\mathbf{v}}\in \Omega  \big|\;\; \mathrm{d}(\widetilde{\mathbf{u}}, \widetilde{\mathbf{v}})< r\right\}.$$

\begin{definition}
 A TDS $(\Omega,\Gamma)$ is called $C_\Gamma$-dynamically uniformly perfect if for every $N\in\mathbb{N}$, there exists $0<C_{N,\Gamma}<1$ such that for every $\widetilde{\mathbf{u}} \in \Omega$ and $r>0$ we have $\mathscr{B}_{d_N}(\widetilde{\mathbf{u}}, r) \setminus \mathscr{B}_{d_N}(\widetilde{\mathbf{u}}, C_{N,\Gamma} r) \neq \emptyset$  and  $C_{\Gamma}=\inf_{N\in\mathbb{N}}\{C_{N,\Gamma} \}>0$. 
\end{definition}

\begin{example}
For any $\widetilde{\mathbf{u}}=(\widetilde{\mathbf{u}}_n)_{n\in \mathbb{N}},\;\widetilde{\mathbf{v}}=(\widetilde{\mathbf{v}}_n)_{n\in \mathbb{N}}\in\mathbb{R}^\mathbb{N}$, let $d(\widetilde{\mathbf{u}},\widetilde{\mathbf{v}})=\displaystyle\sup_{i\in \mathbb{N}}\frac{1}{2^i}\rho(\widetilde{\mathbf{u}}_i,\widetilde{\mathbf{v}}_i)$ where $\rho(u,v)=\min\{|u-v|,1\}$ for any $u,v\in\mathbb{R}$. Consider the shift map $\sigma: (\mathbb{R}^\mathbb{N},d)\to (\mathbb{R}^\mathbb{N},d)$ defined by
\[  \sigma\left((\widetilde{\mathbf{u}}_n)_{n\in \mathbb{N}}\right) = (\widetilde{\mathbf{u}}_{n+1})_{n\in \mathbb{N}}. \]
Let $0<p<1$, obviously, $(\mathbb{R}^\mathbb{N},\sigma)$ is $p$-dynamically uniformly perfect. 
\end{example}

To consider the function $\mathscr{H}^{t}_{\varepsilon}$ on $\mathcal{P}(\Omega)$ for a given set $\Xi \subseteq \Omega$, with $t \geq 0$ and $\varepsilon > 0$, we define it as follows
\[  
\mathscr{H}_{\varepsilon}^{t}(\Xi,d) = \inf\left\{\sum_{i\in\mathbb{N}} (2r_i)^t \:\middle|\: \Xi \subseteq \bigcup_{i\in\mathbb{N}} \mathscr{B}_{d}(\widetilde{\mathbf{u}}_i, r_i) \text{ and } 2r_i < \varepsilon\right\}, \quad \text{for } \Xi \neq \emptyset,  
\]  
and  
\[  
\mathscr{H}_{\varepsilon}^{t}(\emptyset,d) = 0.  
\]  
Subsequently, we set  
\[  
\mathrm{\dim}_{\varepsilon}^{H}(\Xi,d) = \inf \left\{t \geq 0 \mid \overline{\mathscr{H}}_{\varepsilon}^{t}(\Xi,d) \leq 1\right\}.  
\]

\begin{definition}\label{d1}\cite{TM1}
Consider a TDS $(\Omega,\Gamma)$. For any $\Xi\subseteq \Omega$, we define the lower mean Hausdorff dimension and upper mean Hausdorff dimension of $\Xi$ respectively as follows
\begin{align*}  
\underline{\mdim}^{H}(\Xi,d,\Gamma) = \lim_{\varepsilon\to 0}\liminf_{N\to+\infty}\frac{\mathrm{dim}_{\varepsilon}^{H}(\Xi,d_N)}{N}
\end{align*} 
and 
\begin{align*}
\overline{\mdim}^{H}(\Xi,d,\Gamma) = \lim_{\varepsilon\to 0}\limsup_{N\to+\infty}\frac{\mathrm{dim}_{\varepsilon}^{H}(\Xi,d_N)}{N}.  
\end{align*}  
If $\underline{\mdim}^{H}(\Xi,d,\Gamma)$ equals $\overline{\mdim}^{H}(\Xi,d,\Gamma)$, we denote this common value by ${\mdim}^{H}(\Xi,d,\Gamma)$ and refer to it as the mean Hausdorff dimension of $\Xi$ with respect to $\Gamma$ and $d$.
\end{definition}

We define an $(N, d, \Gamma, \varepsilon)$-separated set $\widetilde{\Xi} \subseteq \Xi$ for a compact set $\Xi \subseteq \Omega$, $\varepsilon > 0$, and $N \in \mathbb{N}$, as a subset where the distance $d_{N}(\widetilde{\mathbf{u}}, \widetilde{\mathbf{v}}) \geq \varepsilon$ holds for any two distinct points $\widetilde{\mathbf{u}}, \widetilde{\mathbf{v}} \in \widetilde{\Xi}$. The maximum cardinality among all such $(N, d, \Gamma, \varepsilon)$-separated subsets of $\Xi$ is denoted by $\operatorname{sep}(\Xi, N, d, \Gamma, \varepsilon)$. Furthermore, we introduce the notation  
\[  
\operatorname{sep}(\Xi,d,\Gamma, \varepsilon) = \limsup_{N \rightarrow \infty} \frac{1}{N} \log \operatorname{sep}(\Xi,N,d,\Gamma, \varepsilon).  
\]  
A subset $\Xi^* \subseteq \Xi$ is termed an $(N, d, \Gamma, \varepsilon)$-spanning set of $\Xi$ if for every $\widetilde{\mathbf{u}} \in \Xi$, there exists $\widetilde{\mathbf{v}} \in \Xi^*$ such that $d_{N}(\widetilde{\mathbf{u}}, \widetilde{\mathbf{v}}) < \varepsilon$. The minimum cardinality among all such $(N, d, \Gamma, \varepsilon)$-spanning subsets of $\Xi$ is denoted by $\operatorname{span}(\Xi, N, d, \Gamma, \varepsilon)$. Additionally, we define  
\[  
\operatorname{span}(\Xi,d,\Gamma, \varepsilon) = \limsup_{N \rightarrow \infty} \frac{1}{N} \log \operatorname{span}(\Xi,N,d,\Gamma, \varepsilon).  
\]  
According to \cite[Lemma 2.1]{34}, it holds that  
\[  
\operatorname{sep}(\Xi,N,d,\Gamma, 2\varepsilon) \leq \operatorname{span}(\Xi,N,d,\Gamma, \varepsilon) \leq \operatorname{sep}(\Xi,N,d,\Gamma, \varepsilon),  
\]  
which subsequently implies the inequality  
\begin{equation}\label{ss}  
\operatorname{sep}(\Xi,d,\Gamma, 2\varepsilon) \leq \operatorname{span}(\Xi,d,\Gamma, \varepsilon) \leq \operatorname{sep}(\Xi,d,\Gamma, \varepsilon).  
\end{equation}

\begin{definition}\label{d2}\cite{TM1}
Consider a TDS $(\Omega, \Gamma)$. For any compact subset $\Xi \subseteq \Omega$, we define the lower and upper metric mean dimensions of the set $\Xi$  as follows
  
\begin{align*}  
\underline{\mdim}^{M}(\Xi,d,\Gamma) = \liminf_{\varepsilon \to 0} \frac{\operatorname{sep}(\Xi,d,\Gamma,\varepsilon)}{\vert \log \varepsilon \vert}
\end{align*} 
and 
\begin{align*}
\overline{\mdim}^{M}(\Xi,d,\Gamma) = \limsup_{\varepsilon \to 0} \frac{\operatorname{sep}(\Xi,d,\Gamma,\varepsilon)}{\vert \log \varepsilon \vert}.  
\end{align*}  
If $\underline{\mdim}^{M}(\Xi,d,\Gamma)$ equals $\overline{\mdim}^{M}(\Xi,d,\Gamma)$, we denote the common value by ${\mdim}^{M}(\Xi,d,\Gamma)$ and refer to it as the metric mean dimension of $\Xi$ with respect to $\Gamma$ and $d$.

\end{definition}
Based on \eqref{ss}, the lower and upper metric mean dimensions of $\Xi$ can also be defined as follows
\begin{align*}
\underline{\mdim}^{M}(\Xi,d,\Gamma)&=\liminf_{\varepsilon\to 0}\frac{\operatorname{span}(\Xi,d,\Gamma, \varepsilon)}{\vert \log\varepsilon \vert}\end{align*} and \begin{align*}
\overline{\mdim}^{M}(\Xi,d,\Gamma)&=\limsup_{\varepsilon\to 0}\frac{\operatorname{span}(\Xi,d,\Gamma, \varepsilon)}{\vert \log\varepsilon \vert}.
\end{align*}

\begin{definition}
For $0<\gamma\leq 1$, a function $\Psi:(0,\gamma)\to \mathbb{R}^{+}$ is admissible if $\Psi$ satisfies the following conditions,\vspace{0.2cm}
\begin{enumerate}
\item[(a)] $\Psi$ is monotonic,\vspace{0.2cm}
\item[(b)] $\Psi(\varepsilon)<\varepsilon$ for all $\varepsilon\in (0,\gamma)$,\vspace{0.2cm}
\item[(c)] $\displaystyle\lim_{\varepsilon\to 0^{+}}\frac{\Psi(\varepsilon)}{\varepsilon}=0$.
\end{enumerate}
\end{definition}

\begin{definition}\label{d4}
Let $(\Omega,\Gamma)$ be a TDS, $\Xi\subseteq \Omega$ and $\Psi$ be an admissible function, then we define the upper mean $\Psi$-intermediate dimension as
\begin{align*}
\overline{\mdim}_{\Psi}(\Xi,d,\Gamma)=\inf&\Bigg\{ s\geq 0\mid\text{For any }0<\delta<1\text{ there is }0<\varepsilon(\delta)<1   \text{ such that for any }0<\varepsilon<\varepsilon(\delta),\\
&\text{ there is }N(\delta,\varepsilon)\in\mathbb{N}\text{ which satisfies that for any }N>N(\delta,\varepsilon)\text{ there is a cover } \\&\{\mathscr{B}_{d_N}(\widetilde{\mathbf{u}}_i,r_i)\}_{i\in\mathbb{N}}
\text{ of }\Xi\text{ with }\Psi(\varepsilon)< 2r_i\leq\varepsilon \text{ for every }i\in\mathbb{N} \text{ and } \sum_{i\in\mathbb{N}}(2r_i)^{Ns}<\delta\Bigg\},
\end{align*}
and we define the lower mean $\Psi$-intermediate dimension as
\begin{align*}
\underline{\mdim}_{\Psi}(\Xi,d,\Gamma)=\inf\Bigg\{ &s\geq 0 \mid \text{For any }0<\delta<1,  0<\varepsilon_0 <1
\text{ and }N_0 \in\mathbb{N},
\text{ there exist }0<\varepsilon<\varepsilon_0,\\ &N>N_0
\text{ and a cover }
\{\mathscr{B}_{d_{N}}(\widetilde{\mathbf{u}}_i,r_i)\}_{i\in\mathbb{N}}
\text{ of } \Xi \text{ with }\Psi(\varepsilon)< 2r_i\leq\varepsilon
\text{ for }i\in\mathbb{N}\\
&\text{ and }
\sum_{i\in\mathbb{N}}(2r_i)^{Ns}<\delta\Bigg\}.
\end{align*}
If $\underline{\mdim}_{\Psi}(\Xi,d,\Gamma)$ equals $\overline{\mdim}_{\Psi}(\Xi,d,\Gamma)$, we denote the common value by ${\mdim}_{\Psi}(\Xi,d,\Gamma)$ and refer to it as the metric mean dimension of $\Xi$ with respect to $\Gamma$ and $d$.
\end{definition}


\begin{definition}\label{d6}
Let $(\Omega,\Gamma)$ be a TDS and $\Xi\subseteq \Omega$ be compact, then mean Assouad dimension is defined by
\begin{align*}
\mdim^{A}(\Xi,d,\Gamma)=\inf&\Bigg\{a:\text{There exist }C >0\text{ and }N(C)\in\mathbb{N} \text{ such that for all }
N>N(C),\\
&\operatorname{span}(\Xi\cap \mathscr{B}_{d_N}(\widetilde{\mathbf{u}},R), N,d,\Gamma, r)\leq C \left(\frac{R}{r}\right)^{a}\text{ for all }
\widetilde{\mathbf{u}}\in \Xi, R>0\\
&\hspace{5cm}\text{ and }0<r\leq\min\{1,R\}\Bigg\}.
\end{align*}
\end{definition}

\section{Equivalent definitions of mean dimensions}\label{Section3}
In this section, we provide equivalent definitions for two different mean dimensions. We begin with an equivalent definition of the mean Hausdorff dimension.
\begin{proposition}\label{l1}
Let $(\Omega,\Gamma)$ be a TDS and $\Xi\subseteq \Omega$, then
\begin{align*}
\overline{\mdim}^{H}(\Xi,d,\Gamma)=\inf\Bigg\{ &s\geq 0 \mid \text{For any }0<\delta<1, \text{ there is }0<\varepsilon(\delta)<1 \text{ such that for any }0<\varepsilon<\varepsilon(\delta),\\
&\text{ there is }N(\delta,\varepsilon)\in\mathbb{N} \text{ which satisfies that for any }N>N(\delta,\varepsilon)
\text{ there is a cover }\\&\{\mathscr{B}_{d_N}(\widetilde{\mathbf{u}}_i,r_i)\}_{i\in\mathbb{N}}
\text{ of } \Xi \text{ with } 0<2r_i\leq\varepsilon\text{ for }i\in\mathbb{N} \text{ and }\sum_{i\in\mathbb{N}}(2r_i)^{N(s+\delta)}< 1\Bigg\}
\end{align*}
and
\begin{align*}
\underline{\mdim}^{H}(\Xi,d,\Gamma)=\inf\Bigg\{ &s\geq 0 \mid \text{For any }0<\delta<1, 0<\varepsilon_0 <1
\text{ and }N_0\in\mathbb{N}
\text{ there exist }0<\varepsilon<\varepsilon_0, N>N_0\\
&\text{ and a cover }
\{\mathscr{B}_{d_{N}}(\widetilde{\mathbf{u}}_i,r_i)\}_{i\in\mathbb{N}}
\text{ of } \Xi \text{ with } 0<2r_i\leq\varepsilon
\text{ for }i\in\mathbb{N}
\text{ and }\\
&\hspace{5cm}\sum_{i\in\mathbb{N}}(2r_i)^{N(s+\delta)}<1\Bigg\}.
\end{align*}
\end{proposition}
\begin{proof}
Let
\begin{align*}
up=\inf\Bigg\{ &s\geq 0 \mid \text{For any }0<\delta<1, \text{ there is }0<\varepsilon(\delta)<1 \text{ such that for any }0<\varepsilon<\varepsilon(\delta),\\
&\text{ there is }N(\delta,\varepsilon)\in\mathbb{N} \text{ which satisfies that for any }N>N(\delta,\varepsilon)
\text{ there is a cover }\\&\{\mathscr{B}_{d_N}(\widetilde{\mathbf{u}}_i,r_i)\}_{i\in\mathbb{N}} 
\text{ of } \Xi \text{ with } 0<2r_i\leq\varepsilon\text{ for }i\in\mathbb{N} \text{ and }\sum_{i\in\mathbb{N}}(2r_i)^{N(s+\delta)}< 1\Bigg\}
\end{align*}
and
\begin{align*}
low=\inf\Bigg\{ &s\geq 0 \mid \text{For any }0<\delta<1, 0<\varepsilon_0 <1
\text{ and }N_0 \in\mathbb{N}
\text{ there exist }0<\varepsilon<\varepsilon_0, N>N_0\\
&\text{ and a cover }
\{\mathscr{B}_{d_{N}}(\widetilde{\mathbf{u}}_i,r_i)\}_{i\in\mathbb{N}}
\text{ of } \Xi \text{ with } 0<2r_i\leq\varepsilon
\text{ for }i\in\mathbb{N}
\text{ and }\\
&\hspace{5cm}\sum_{i\in\mathbb{N}}(2r_i)^{N(s+\delta)}<1\Bigg\}.
\end{align*}

\item[(a)]First we prove the equivalence definition of the upper mean Hausdorff dimension. Let $t<up$,
then there exists $0<\delta<1$ such that for any $0<\varepsilon_0<1$, there exists $0<\varepsilon\leq\varepsilon_0$ which satisfies that for any $N_0>0$ there exists $N\geq N_0$ such that for any cover $\{\mathscr{B}_{d_N}(\widetilde{\mathbf{u}}_i,r_i)\}_{i\in\mathbb{N}}$ of $\Xi$ with $0<2r_i\leq\varepsilon$ for $i\in\mathbb{N}$, we have
$$
\sum_{i\in\mathbb{N}}(2r_i)^{N(t+\delta)}\geq 1.
$$
This implies that for any $0<\varepsilon_0<1$, there exists an increasing sequence $\{N_i(\varepsilon_0)\}_{i\in\mathbb{N}}\subseteq\mathbb{N}$ with $\lim_{i\to\infty}N_i(\varepsilon_0)=\infty$ such that
for any $i\in\mathbb{N}$,
$$N_i(\varepsilon_0) (t+\delta)\leq\mathrm{\dim}_{\varepsilon_0}^{H}(\Xi,d_{N_i(\varepsilon_0)}).
$$
Therefore
\begin{align*}
\overline{\mdim}^{H}(\Xi,d,\Gamma)&=\lim_{\varepsilon_0 \to 0}\limsup_{N \to\infty}\frac{\mathrm{\dim}_{\varepsilon_0}^{H}(\Xi,d_{N})}{N}\\
&\geq \lim_{\varepsilon_0 \to 0}\limsup_{i \to\infty}\frac{\mathrm{\dim}_{\varepsilon_0}^{H}(\Xi,d_{N_{i}(\varepsilon_0)})}{N_{i}(\varepsilon_0)}
\geq t+\delta\geq t.
\end{align*}
Assume that $t>up$,
then for any $0<\delta<1$, there is a $0<\varepsilon(\delta)<1$ such that for any $\varepsilon\in(0,\varepsilon(\delta))$ there is a $N(\delta,\varepsilon)>0$ which satisfies that for any $N>N(\delta,\varepsilon)$ there is a cover $\{\mathscr{B}_{d_N}(\widetilde{\mathbf{u}}_i,r_i)\}_{i\in\mathbb{N}}$ of $\Xi$ with $0<2r_i\leq\varepsilon$ for every $i\in\mathbb{N}$ and
$$
\sum_{i\in\mathbb{N}}(2r_i)^{N(t+\delta)}< 1.
$$
Thus, $\dim_{\varepsilon}^{H}(\Xi,d_N)\leq N(t+\delta)$ and
$$
\overline{\mdim}^{H}(\Xi,d,\Gamma)=\lim_{\varepsilon\to 0}\limsup_{N\to+\infty}\frac{\mathrm{\dim}_{\varepsilon}^{H}(\Xi,d_N)}{N} \leq t+\delta.
$$
Since $\delta$ is arbitrary, then
$$
\overline{\mdim}^{H}(\Xi,d,\Gamma)\leq t.
$$

\item[(b)] Next we prove the equivalence definition of the lower mean Hausdorff dimension. Let $t<low$,
then there exist $0<\delta<1$ and $0<\varepsilon(\delta)<1$ such that for any $0<\varepsilon<\varepsilon(\delta)$, there is $N(\delta,\varepsilon)>0$ which satisfies for any $N>N(\delta,\varepsilon)$ and any cover $\{\mathscr{B}_{d_N}(\widetilde{\mathbf{u}}_i,r_i)\}_{i\in\mathbb{N}}$ of $\Xi$ with $0<2r_i\leq\varepsilon$ for $i\in\mathbb{N}$, we have
$$
\sum_{i\in\mathbb{N}}(2r_i)^{N(t+\delta)}\geq 1.
$$
Thus
$$
N(t+\delta)\leq\mathrm{\dim}_{\varepsilon}^{H}(\Xi,d_N)
$$
and
$$
\underline{\mdim}^{H}(\Xi,d,\Gamma)=\lim_{\varepsilon\to 0}\liminf_{N\to+\infty}\frac{\mathrm{\dim}_{\varepsilon}^{H}(\Xi,d_N)}{N}\geq t+\delta\geq t.
$$

Let $t>low$,
then for any $0<\delta<1$, $0<\varepsilon_0 <1$ and $N_0>0$ there exist $0<\varepsilon\leq\varepsilon_0$, $N\geq N_0$ and a cover $\{\mathscr{B}_{d_{N}}(\widetilde{\mathbf{u}}_i,r_i)\}_{i\in\mathbb{N}}$
of $\Xi$ with $0<2r_i\leq\varepsilon$ for every $i\in\mathbb{N}$ such that $\sum_{i\in\mathbb{N}}(2r_i)^{N(t+\delta)}<1$. 

This implies that for any $0<\varepsilon_0<1$, there exists an increasing sequence $\{N_i(\varepsilon_0)\}_{i\in\mathbb{N}}\subseteq\mathbb{N}$ with $\lim_{i\to\infty}N_i(\varepsilon_0)=\infty$ such that
for any $i\in\mathbb{N}$,
$$\dim_{\varepsilon_0}^{H}(\Xi,d_{N_i(\varepsilon_0)})\leq N_i(\varepsilon_0) (t+\delta),$$ which gives that
\begin{align*}
\underline{\mdim}^{H}(\Xi,d,\Gamma)&=\lim_{\varepsilon_0 \to 0}\liminf_{N \to\infty}\frac{\dim_{\varepsilon_0}^{H}(\Xi,d_{N})}{N}\\
&\leq\lim_{\varepsilon_0 \to 0}\liminf_{i\to\infty}\frac{\dim_{\varepsilon_0}^{H}(\Xi,d_{N_i(\varepsilon_0)})}{N_i(\varepsilon_0)}\leq t+\delta.
\end{align*}
Since $\delta$ is arbitrary, then
$$
\underline{\mdim}^{H}(\Xi,d,\Gamma)\leq t.
$$
\end{proof}

Next, we give the equivalent definition of metric mean dimension.
\begin{proposition}\label{l2}
Let $(\Omega,\Gamma)$ be a TDS and $\Xi\subseteq \Omega$ be a compact subset, then
\begin{align*}
\overline{\mdim}^{M}(\Xi,d,\Gamma)=\inf&\Bigg\{ s\geq 0\mid\text{For any }0<\delta<1\text{ there is }0<\varepsilon(\delta) <1  \text{ such that for any }\varepsilon\in(0,\varepsilon(\delta)),\\
&\text{ there is }N(\delta,\varepsilon)\in\mathbb{N}\text{ which satisfies that for any }N>N(\delta,\varepsilon)\text{ there is a cover }\\ 
&\{\mathscr{B}_{d_N}(\widetilde{\mathbf{u}}_i,r_i)\}_{i}
\text{ of }\Xi\text{ with }2r_i=\varepsilon \text{ for every }i \text{ and } \sum_{i}(2r_i)^{Ns}<\delta\Bigg\}
\end{align*}
and
\begin{align*}
\underline{\mdim}^{M}(\Xi,d,\Gamma)=\inf&\Bigg\{ s\geq 0\mid\text{For any }0<\delta<1\text{ and }0<\varepsilon_0 <1 \text{ there exist }\varepsilon\in(0,\varepsilon_0) \text{ and }\\
&N(\delta,\varepsilon)\in\mathbb{N}\text{ which satisfies for any }N>N(\delta,\varepsilon)\text{ there is a cover } \{\mathscr{B}_{d_N}(\widetilde{\mathbf{u}}_i,r_i)\}_{i}\\
&\hspace{3cm}\text{ of }\Xi\text{ with }2r_i=\varepsilon \text{ for every }i \text{ and } \sum_{i}(2r_i)^{Ns}<\delta\Bigg\}.
\end{align*}
\end{proposition}
\begin{proof}
Let
\begin{align*}
    up=\inf&\Bigg\{ s\geq 0\mid\text{For any }0<\delta<1\text{ there is }0<\varepsilon(\delta) <1  \text{ such that for any }\varepsilon\in(0,\varepsilon(\delta)),\\
&\text{ there is }N(\delta,\varepsilon)\in\mathbb{N}\text{ which satisfies that for any }N>N(\delta,\varepsilon)\text{ there is a cover }\\ 
&\{\mathscr{B}_{d_N}(\widetilde{\mathbf{u}}_i,r_i)\}_{i}
\text{ of }\Xi\text{ with }2r_i=\varepsilon \text{ for every }i \text{ and } \sum_{i}(2r_i)^{Ns}<\delta\Bigg\}
\end{align*}
and
\begin{align*}
    low=\inf&\Bigg\{ s\geq 0\mid\text{For any }0<\delta<1\text{ and }0<\varepsilon_0 <1 \text{ there exist }\varepsilon\in(0,\varepsilon_0) \text{ and }\\
&N(\delta,\varepsilon)\in\mathbb{N}\text{ which satisfies for any }N>N(\delta,\varepsilon)\text{ there is a cover } \{\mathscr{B}_{d_N}(\widetilde{\mathbf{u}}_i,r_i)\}_{i}\\
&\hspace{3cm}\text{ of }\Xi\text{ with }2r_i=\varepsilon \text{ for every }i \text{ and } \sum_{i}(2r_i)^{Ns}<\delta\Bigg\}.
\end{align*}

\item[(a)] Initially, we establish the equivalence of the definition for the upper metric mean dimension. Let $t>up$,
then for any $\delta>0$, there is $0<\varepsilon(\delta)<1$ such that for any $0<\varepsilon<\varepsilon(\delta)$, there is $N(\delta,\varepsilon)>0$ such that for any $N>N(\delta,\varepsilon)$ there is a cover $\{\mathscr{B}_{d_N}(\widetilde{\mathbf{u}}_i,r_i)\}_{i\in\mathbb{N}}$ of $\Xi$ with $2r_i=\varepsilon$ for every $i\in\mathbb{N}$ such that $\operatorname{span}(\Xi,N,d,\Gamma, \varepsilon)\varepsilon^{Nt}<\delta$, then we obtain
\begin{align*}
\frac{\log\operatorname{span}(\Xi,N,d,\Gamma, \varepsilon)}{N}+t\log\varepsilon\leq\frac{\log\delta}{N}
\end{align*}
and
\begin{align*}
\limsup_{\varepsilon\to 0}\limsup_{N\to\infty}\frac{\log\operatorname{span}(\Xi,N,d,\Gamma, \varepsilon)}{-N\log\varepsilon}\leq t.
\end{align*}

On the other hand, let $t>\overline{\mdim}^{M}(\Xi,d,\Gamma)$. For every $\delta>0$, we have
\begin{align*}
\liminf_{\varepsilon\to 0}\limsup_{N\to \infty}\Big(t-\frac{\log\delta}{N\log\varepsilon}\Big)>\limsup_{\varepsilon\to 0}\limsup_{N\to\infty}\frac{\log\operatorname{span}(\Xi,N,d,\Gamma, \varepsilon)}{-N\log\varepsilon}.
\end{align*}
Then there is  $0<\varepsilon(\delta)<1$ such that for every $0<\varepsilon<\varepsilon(\delta)$,
\begin{align*}
\limsup_{N\to \infty}\Big(t-\frac{\log\delta}{N\log\varepsilon}\Big)>\limsup_{N\to\infty}\frac{\log\operatorname{span}(\Xi,N,d,\Gamma, \varepsilon)}{-N\log\varepsilon}.
\end{align*}
This implies that there is $N(\delta,\varepsilon)>0$ such that for every $N>N(\delta,\varepsilon)$, we have
\begin{align*}
t-\frac{\log\delta}{N\log\varepsilon}>\frac{\log\operatorname{span}(\Xi,N,d,\Gamma, \varepsilon)}{-N\log\varepsilon}.
\end{align*}
Thus $\operatorname{span}(\Xi,N,d,\Gamma, \varepsilon)~\varepsilon^{Nt}<\delta$, which implies that there is a cover $\{\mathscr{B}_{d_N}(\widetilde{\mathbf{u}}_i,r_i)\}_{i\in\mathbb{N}}$ of $\Xi$ with $2r_i=\varepsilon$ for every $i\in\mathbb{N}$ such that $\sum_{i\in\mathbb{N}}(2r_i)^{Nt}<\delta$. So we can deduce that $t\geq up$.

\item[(b)]Afterward, we demonstrate the equivalence of the definition for the lower metric mean dimension. Suppose that $t>low$,
then for any $0<\delta<1$ and $0<\varepsilon_0 <1$, there exists $0<\varepsilon\leq\varepsilon_0$ and $N(\delta,\varepsilon)>0$ such that for any $N>N(\delta,\varepsilon)$ there exist a cover $\{\mathscr{B}_{d_N}(\widetilde{\mathbf{u}}_i,r_i)\}_{i\in\mathbb{N}}$ of $\Xi$ with $2r_i=\varepsilon$ for $i\in\mathbb{N}$ such that $\operatorname{span}(\Xi,N,d,\Gamma, \varepsilon)~\varepsilon^{Nt}<\delta$, then
\begin{align*}
\frac{\log\operatorname{span}(\Xi,N,d,\Gamma, \varepsilon_0)}{N}+t\log\varepsilon_0 \leq\frac{\log\delta}{N}
\end{align*}
and
\begin{align*}
\liminf_{\varepsilon\to 0}\limsup_{N\to\infty}\frac{\log\operatorname{span}(\Xi,N,d,\Gamma, \varepsilon_0)}{-N\log\varepsilon_0}\leq t.
\end{align*}

On the other hand, let $t<low$,
then there exist $0<\delta<1$ and $0<\varepsilon(\delta) <1$, for any $0<\varepsilon<\varepsilon(\delta)$ and $N_0 >0$, there is a $N\geq N_0$ such that for any cover $\{\mathscr{B}_{d_N}(\widetilde{\mathbf{u}}_i,r_i)\}_{i\in\mathbb{N}}$ of $\Xi$ with $2r_i=\varepsilon$ for $i\in\mathbb{N}$, we have $\operatorname{span}(\Xi,N,d,\Gamma, \varepsilon)~\varepsilon^{Nt}\geq \delta$. 

This implies that for any $0<\varepsilon<\varepsilon(\delta)$, there exists an increasing sequence $\{N_i(\varepsilon)\}_{i\in\mathbb{N}}\subseteq\mathbb{N}$ with $\lim_{i\to\infty}N_i(\varepsilon)=\infty$ such that
for any $i\in\mathbb{N}$,
\begin{align*}
\frac{\log\operatorname{span}(\Xi,N_i(\varepsilon),d,\Gamma, \varepsilon)}{N_i(\varepsilon)}+t\log\varepsilon\geq \frac{\log\delta}{N_i(\varepsilon)}
\end{align*}
and
\begin{align*}
\underline{\mdim}^{M}(\Xi,d,\Gamma)&=\liminf_{\varepsilon\to 0}\limsup_{N \to\infty}\frac{\log\operatorname{span}(\Xi,N ,d,\Gamma, \varepsilon)}{-N\log\varepsilon}\\
&\geq\liminf_{\varepsilon\to 0}\limsup_{i \to\infty}\frac{\log\operatorname{span}(\Xi,N_i(\varepsilon) ,d,\Gamma, \varepsilon)}{-N_i(\varepsilon)\log\varepsilon}\geq t.
\end{align*}
\end{proof}

Finally, we present an equivalent characterization of the mean $\Psi$-intermediate dimension of the space $\Omega$. 
Let $\diam_{d}(\Xi)$ denote the diameter of $\Xi$ with respect to the metric $d$. Assume that $\Psi$ represents an admissible function and that $(\Omega, \Gamma)$ is dynamically uniformly perfect. We suppose that $\mdim^{A}(\Omega, d, \Gamma) < +\infty$.

If $\overline{\mdim}_{\Psi}(\Omega,d,\Gamma)>0$, then
\begin{align*}
\overline{\mdim}_{\Psi}(\Omega,d,\Gamma)=\sup\Bigg\{&s>0 \mid \text{ there exists } C\in (0,+\infty) \text{ such that for all } 0<\varepsilon_0 <1 \text{ and } N_0 \in\mathbb{N}, \\
&\text{ there exist } 0<\varepsilon\leq \varepsilon_0, N>N_0 \text{ and a measure } \nu_{\varepsilon,N}  
\text{ such that if } U \text{ is a Borel }\\
&\text{ subset of } \Omega \text{ which satisfies } \Psi(\varepsilon)<\diam_{d_N}(U)\leq \varepsilon
\text{ then } \\&\nu_{\varepsilon,N}(U)\leq C \diam_{d_N}(U)^{Ns}\Bigg\}.
\end{align*}

If $\underline{\mdim}_{\Psi}(\Omega,d,\Gamma)>0$, then
\begin{align*}
\underline{\mdim}_{\Psi}(\Omega,d,\Gamma)=\sup\Bigg\{&s>0 \mid \text{ there exists } C\in (0,+\infty) \text{ such that there exists  } 0<\varepsilon(C) <1 \text{ for any }\\
&0<\varepsilon<\varepsilon(C)\text{ there exists } N(C,\varepsilon) \in\mathbb{N} \text{ which satisfies for any }N>N(C,\varepsilon) ,\\ &\text{ there exists }
\text{ a measure }\nu_{\varepsilon,N} 
\text{ such that if } U \text{ is a Borel subset of } \Omega \text{ with}\\  & \Psi(\varepsilon)<\diam_{d_N}(U)\leq \varepsilon
\text{ then } \nu_{\varepsilon,N}(U)\leq C \diam_{d_N}(U)^{Ns}\Bigg\}.
\end{align*}

Currently, we cannot offer proof for this equivalent definition as it necessitates employing the mass distribution principle and a Frostman-type lemma. We will finalize this proof in the upcoming sections; refer to Proposition \ref{p343}.

\section{Some properties of the mean $\Psi$-intermediate dimensions}\label{Section4}
\subsection{Basic properties} We begin by exploring fundamental characteristics of mean intermediate dimensions, which are commonly encountered in various dimension definitions. It should be emphasized that the proof methods in this section are mainly based on \cite{BA} and \cite{13}.
\begin{lemma}\label{l3}
Let $(\Omega,\Gamma)$ be a TDS and $\Xi$ be a compact subset of $\Omega$, then
\begin{align*}
0\leq \underline{\mdim}^{H}(\Xi,d,\Gamma)\leq \overline{\mdim}^{H}(\Xi,d,\Gamma)
\leq \overline{\mdim}_{\Psi}(\Xi,d,\Gamma)\leq \overline{\mdim}^{M}(\Xi,d,\Gamma)
\end{align*}
and
\begin{align*}
0\leq\underline{\mdim}^{H}(\Xi,d,\Gamma)\leq\underline{\mdim}_{\Psi}(\Xi,d,\Gamma)\leq \underline{\mdim}^{M}(\Xi,d,\Gamma)\leq \overline{\mdim}^{M}(\Xi,d,\Gamma).
\end{align*}
\end{lemma}
\begin{proof}

First we prove $\overline{\mdim}^{H}(\Xi,d,\Gamma)\leq \overline{\mdim}_{\Psi}(\Xi,d,\Gamma)$. 
Take $\psi>\overline{\mdim}_{\Psi}(\Xi,d,\Gamma)$, then for any $0<\delta<1$ there is $0<\varepsilon(\delta)<1$ such that for any $\varepsilon\in(0,\varepsilon(\delta))$, there is $N(\delta,\varepsilon)\in\mathbb{N}$ which satisfies for any $N>N(\delta,\varepsilon)$ there is a cover $\{\mathscr{B}_{d_N}(\widetilde{\mathbf{u}}_i,r_i)\}_{i\in\mathbb{N}}$ of $\Xi$ with $\Psi(\varepsilon)< 2r_i\leq\varepsilon$ for every $i\in\mathbb{N}$ such that $\sum_{i\in\mathbb{N}}(2r_i)^{N\psi}<\delta $. Then
\begin{align*}
    \sum_{i\in\mathbb{N}}(2r_i)^{N(\psi+\delta)}=\sum_{i\in\mathbb{N}}(2r_i)^{N\psi}(2r_i)^{N\delta}\leq \sum_{i\in\mathbb{N}}(2r_i)^{N\psi}<\delta<1.
\end{align*}
Thus, $\overline{\mdim}^{H}(\Xi,d,\Gamma)\leq \psi$ and $\overline{\mdim}^{H}(\Xi,d,\Gamma)\leq \overline{\mdim}_{\Psi}(\Xi,d,\Gamma)$.

The proof of $\underline{\mdim}^{H}(\Xi,d,\Gamma)\leq \underline{\mdim}_{\Psi}(\Xi,d,\Gamma)$ is very similar and has been taken out.

The other inequalities are direct consequences of Lemmas \ref{l1}, \ref{l2}, and Definition \ref{d4}.
\end{proof}

\begin{lemma}
Let $(\Omega,\Gamma)$ be a TDS and $\Xi$ be a compact subset of $\Omega$. If $\Gamma(\Xi)\subseteq \Xi$, then
$$\overline{\mdim}^{M}(\Xi,d,\Gamma)\leq \mdim^{A}(\Xi,d,\Gamma).$$
\end{lemma}
\begin{proof}
Let $a>\mdim^{A}(\Xi,d,\Gamma)$ and $R=\diam_{d}(\Xi)$. Since $\Gamma(\Xi)\subseteq \Xi$, $\diam_{d_N}(\Xi)\leq R$ holds for every $N\in\mathbb{N}$. Then there exists $N>N(C)$ such that for any $\widetilde{\mathbf{u}}\in\Xi$ and $r\leq \min\{1,R\}$,
\begin{align*}
\operatorname{span}(\Xi\cap \mathscr{B}_{d_N}(\widetilde{\mathbf{u}},R), N,d,\Gamma, r)=
\operatorname{span}(\Xi, N,d,\Gamma, r)\leq C\Big(\frac{R}{r}\Big)^{a}.
\end{align*}
Thus,
\begin{align*}
\log\operatorname{span}(\Xi, N,d,\Gamma, r)\leq C\Big(\frac{R}{r}\Big)^{a}&\leq a\log R +a\log C- a\log r\\
&\leq a\log R +a\log C- aN\log r.
\end{align*}
Then
\begin{align*}
\frac{\log \operatorname{span}(\Xi\cap \mathscr{B}_{d_N}(\widetilde{\mathbf{u}},R), N,d,\Gamma, r)}{-N\log r} &\leq a\frac{\log R +\log C}{-N\log r}+a,
\end{align*}
which implies that
\begin{align*}
\limsup_{r\to 0}\limsup_{N\to\infty}\frac{\log \operatorname{span}(\Xi\cap \mathscr{B}_{d_N}(\widetilde{\mathbf{u}},R), N,d,\Gamma, r)}{-N\log r}\leq a,
\end{align*}
i.e. $\overline{\mdim}^{M}(\Xi,d,\Gamma)\leq \mdim^{A}(\Xi,d,\Gamma)$.
\end{proof}

The lemma below follows directly from Definition \ref{d4}.
\begin{lemma}\label{l302}
Let $(\Omega,\Gamma)$ be a TDS, $\Xi,\widetilde{\Xi}\subseteq \Omega$ and $\Psi,\Psi_1$ be two admissible functions. Then the following results hold.
\begin{enumerate}
\item[(a)] If $\Xi\subseteq \widetilde{\Xi}$, then $$\overline{\mdim}_{\Psi}(\Xi,d,\Gamma)\leq \overline{\mdim}_{\Psi}(\widetilde{\Xi},d,\Gamma)$$ and $$\underline{\mdim}_{\Psi}(\Xi,d,\Gamma)\leq \underline{\mdim}_{\Psi}(\widetilde{\Xi},d,\Gamma).$$
\item[(b)] If $\Psi\leq \Psi_1$, then $$\overline{\mdim}_{\Psi}(\Xi,d,\Gamma)\leq \overline{\mdim}_{\Psi_1}(\Xi,d,\Gamma)$$ and $$\underline{\mdim}_{\Psi}(\Xi,d,\Gamma)\leq \underline{\mdim}_{\Psi_1}(\Xi,d,\Gamma).$$
\end{enumerate}
\end{lemma}

We now introduce the finite stability of the mean intermediate dimensions.
\begin{lemma}\label{l303}
Let $(\Omega,\Gamma)$ be a TDS, $\Xi,\widetilde{\Xi}\subseteq \Omega$ and $\Psi$ be an admissible function, then
\begin{enumerate}
\item[(a)]
$$
\overline{\mdim}_{\Psi}(\Xi\cup \widetilde{\Xi},d,\Gamma)=\max\left\{\overline{\mdim}_{\Psi}(\Xi,d,\Gamma),\overline{\mdim}_{\Psi}(\widetilde{\Xi},d,\Gamma)\right\},
$$
\item[(b)]
\begin{align*}
\max\left\{\underline{\mdim}_{\Psi}(\Xi,d,\Gamma),\underline{\mdim}_{\Psi}(\widetilde{\Xi},d,\Gamma)\right\}&\leq\underline{\mdim}_{\Psi}(\Xi\cup \widetilde{\Xi},d,\Gamma)\\
&\leq\max\left\{\overline{\mdim}_{\Psi}(\Xi,d,\Gamma),\underline{\mdim}_{\Psi}(\widetilde{\Xi},d,\Gamma)\right\}.
\end{align*}
\end{enumerate}
\end{lemma}
\begin{proof}
First we prove $(a)$. Obviously, 
$$
\overline{\mdim}_{\Psi}(\Xi\cup \widetilde{\Xi},d,\Gamma)\geq\max\left\{\overline{\mdim}_{\Psi}(\Xi,d,\Gamma),\overline{\mdim}_{\Psi}(\widetilde{\Xi},d,\Gamma)\right\}.
$$
Suppose $t > \overline{\mdim}_{\Psi}(\Xi,d,\Gamma)$ and $s > \overline{\mdim}_{\Psi}(\widetilde{\Xi},d,\Gamma)$. Without loss of generality, let us assume $t > s$, which implies $t = \max\{s,t\}$. For any $0<\delta<1$, there exist $0<\varepsilon (\delta)<1$ such that for any $0<\varepsilon<\varepsilon(\delta)$, there is $N(\delta,\varepsilon)\in\mathbb{N}$ which satisfies for any $N>N(\delta,\varepsilon)$, there exists a cover $\{\mathscr{B}_{d_N}(\widetilde{\mathbf{u}}_i,r_i)\}_{i}$ of $\Xi$ with $\Psi(\varepsilon) < 2r_i \leq \varepsilon$ for all $i$, and a cover $\{\mathscr{B}_{d_N}(\widetilde{\mathbf{v}}_j,l_j)\}_{j}$ of $\widetilde{\Xi}$ with $\Psi(\varepsilon) < 2l_j \leq \varepsilon$ for all $j$, such that
$$
\sum_{i}(2r_i)^{Nt}< \delta
\quad\text{and}\quad
\sum_{j}(2l_j)^{Ns}< \delta .
$$
It follows from
$$
\Xi\cup \widetilde{\Xi}\subseteq \bigcup_{i}\mathscr{B}_{d_N}(\widetilde{\mathbf{u}}_i,r_i)\cup\bigcup_{j}\mathscr{B}_{d_N}(\widetilde{\mathbf{v}}_j,l_j),
$$
that
\begin{align*}
\sum_{i}(2r_i)^{N\max\{s,t\}}+\sum_{j}(2l_j)^{N\max\{s,t\}}&=\sum_{i}(2r_i)^{Nt}+\sum_{j}(2l_j)^{Nt}\\
&\leq \sum_{i}(2r_i)^{Nt}+\sum_{j}(2l_j)^{Ns}\\
&< 2\delta.
\end{align*}
Due to the arbitrariness of $\delta$, $\overline{\mdim}_{\Psi}(\Xi\cup \widetilde{\Xi},d,\Gamma)\leq t$ and $$\overline{\mdim}_{\Psi}(\Xi\cup \widetilde{\Xi},d,\Gamma)\leq \max\left\{\overline{\mdim}_{\Psi}(\Xi,d,\Gamma),\overline{\mdim}_{\Psi}(\widetilde{\Xi},d,\Gamma)\right\}.$$

Now we will prove $(b)$. It is clear that
$$
\max\left\{\underline{\mdim}_{\Psi}(\Xi,d,\Gamma),\underline{\mdim}_{\Psi}(\widetilde{\Xi},d,\Gamma)\right\}\leq\underline{\mdim}_{\Psi}(\Xi\cup \widetilde{\Xi},d,\Gamma).
$$
Take $e>\overline{\mdim}_{\Psi}(\Xi,d,\Gamma)$ and $h>\underline{\mdim}_{\Psi}(\widetilde{\Xi},d,\Gamma)$. Let $\delta_1>0$ and $\delta_2>0$, then the following two claims hold. 
\begin{enumerate}
\item [(1)]There exists $0<\varepsilon(\delta_1)<1$, for any $0<\varepsilon<\varepsilon(\delta_1)$, there is a $N(\delta_1,\varepsilon)\in\mathbb{N}$ for all $N>N(\delta_1,\varepsilon)$, there is a cover $\{\mathscr{B}_{d_N}(\widetilde{\mathbf{u}}_k,r_k)\}_{k}$ with $\Psi(\varepsilon)<2r_k\leq \varepsilon$ for all $k$, satisfy
$$
\sum_{k}(2r_k)^{Ne}< \delta_1.
$$
\item [(2)]For all $0<\varepsilon_0<1$ and $N_0\in\mathbb{N}$, there exists $0<\varepsilon<\varepsilon_0$ and $N>N_0$ such that there is a cover $\{\mathscr{B}_{d_N}(\widetilde{\mathbf{v}}_m,l_m)\}_{m}$ with $\Psi(\varepsilon)<2l_m\leq \varepsilon$ for all $m$, satisfy
$$
\sum_{m}(2l_m)^{Nh}< \delta_2.
$$
\end{enumerate}

For $0<\delta<1$, there exists $0<\varepsilon(\delta)<1$, for any $0<\varepsilon<\varepsilon(\delta)$, there is a $N(\delta,\varepsilon)\in\mathbb{N}$ such that for all $N>N(\delta,\varepsilon)$, there is a cover $\{\mathscr{B}_{d_N}(\widetilde{\mathbf{u}}_k,r_k)\}_{k}$ with $\Psi(\varepsilon)<2r_k\leq \varepsilon$ for all $k$, satisfies
$$
\sum_{k}(2r_k)^{Ne}< \frac{\delta}{2}.
$$

Choose a positive decreasing sequence $\mathcal{E}=\{\varepsilon_n\}_{n}$ with $\lim_{n\to\infty}\varepsilon_{n}=0$ and $0<\varepsilon_n<\varepsilon(\delta)$ for every $n$, and a positive increasing sequence $\mathcal{N}=\{N_{n}\}_{n}$ with $\lim_{n\to\infty}N_{n}=\infty$ and $N_n >N(\delta,\varepsilon_n)$, such that there is a cover $\{\mathscr{B}_{d_{N_n}}(\widetilde{\mathbf{v}}_m,l_m)\}_{m}$ with $\Psi(\varepsilon_n)<2l_m\leq \varepsilon_n$ for all $m$, satisfies
$$
\sum_{m}(2l_m)^{N_n h}< \frac{\delta}{2}.
$$ 

Set $p=\max\{e,h\}$. It is evident that for any $n\in\mathbb{N}$
$$
\Xi\cup \widetilde{\Xi}\subseteq \bigcup_{k}\mathscr{B}_{d_{N_n}}(\widetilde{\mathbf{u}}_k,r_k)\cup\bigcup_{m}\mathscr{B}_{d_{N_n}}(\widetilde{\mathbf{v}}_m,l_m),
$$
therefore,
\begin{align*}
\sum_{k}(2r_k)^{{N_n}\max\{e,h\}}+\sum_{m}(2l_m)^{{N_n}\max\{e,h\}}&=\sum_{k}(2r_k)^{{N_n}p}+\sum_{m}(2l_m)^{{N_n}p}\\
&\leq \sum_{k}(2r_k)^{{N_n}e}+\sum_{m}(2l_m)^{{N_n}h}\\
&< \delta.
\end{align*}
Due to the arbitrariness of $\delta$, $\overline{\mdim}_{\Psi}(\Xi\cup \widetilde{\Xi},d,\Gamma)\leq p$ and $$\underline{\mdim}_{\Psi}(\Xi\cup \widetilde{\Xi},d,\Gamma)\leq \max\left\{\overline{\mdim}_{\Psi}(\Xi,d,\Gamma),\underline{\mdim}_{\Psi}(\widetilde{\Xi},d,\Gamma)\right\}.$$
\end{proof}
Next, we will give the conditions for the different mean dimensions to be equal.

\begin{lemma}
Let $(\Omega,\Gamma)$ be a TDS, $\Xi\subseteq\Omega$ and $\Psi$ be an admissible function.
\begin{enumerate}
    \item[(a)] If $\Psi=0$, then
    $$
    \overline{\mdim}_{\Psi}(\Xi,d,\Gamma)=\overline{\mdim}^{H}(\Xi,d,\Gamma)
    $$
    and
    $$
    \underline{\mdim}_{\Psi}(\Xi,d,\Gamma)=\underline{\mdim}^{H}(\Xi,d,\Gamma).
    $$
    \item[(b)] If $\Xi$ is compact and $\displaystyle\lim_{\varepsilon\to 0}\frac{\log\varepsilon}{\log\Psi(\varepsilon)}=1$, then
    $$
    \overline{\mdim}_{\Psi}(\Xi,d,\Gamma)=\overline{\mdim}^{M}(\Xi,d,\Gamma)
    $$
    and
    $$
    \underline{\mdim}_{\Psi}(\Xi,d,\Gamma)=\underline{\mdim}^{M}(\Xi,d,\Gamma).
    $$
\end{enumerate}
\end{lemma}
\begin{proof}
The proof of $(a)$ is a direct consequence of Definition \ref{d4} and Lemma \ref{l1}.

For $(b)$, we only prove $\overline{\mdim}_{\Psi}(\Xi,d,\Gamma)=\overline{\mdim}^{M}(\Xi,d,\Gamma)$, because the proof of the other equation is similar. Assume  that $\overline{\mdim}_{\Psi}(\Xi,d,\Gamma)<\overline{\mdim}^{M}(\Xi,d,\Gamma)$, then there exist $s,t>0$ such that $$\overline{\mdim}_{\Psi}(\Xi,d,\Gamma)<s<t<\overline{\mdim}^{M}(\Xi,d,\Gamma)$$ which implies that there exists $0<\varepsilon(1)<1$ such that for all $0<\varepsilon<\varepsilon(1)$, there exists $N(1,\varepsilon(1))\in\mathbb{N}$ which satisfies for all $N>N(1,\varepsilon(1))$, there is a cover $\{\mathscr{B}_{d_N}(\widetilde{\mathbf{u}}_i,r_i)\}_{i}$ of $\Xi$ with $\Psi(\varepsilon)< 2r_i\leq \varepsilon$ and $\sum_{i}(2r_i)^{s}< 1$. Therefore, we obtain
\begin{align*}
\operatorname{span}(\Xi,N,d,\Gamma,\varepsilon)\varepsilon^{Nt}&\leq \sum_{i}\varepsilon^{Nt} \frac{(2r_i)^{Ns}(2r_i)^{N(t-s)}}{(2r_i)^{Nt}}\\
&\leq
\sum_{i}\varepsilon^{Nt}\frac{(2r_i)^{Ns}\varepsilon^{N(t-s)}}{(\Psi(\varepsilon))^{Nt}}\\
&\leq \Big( \frac{\varepsilon^{1+\frac{t-s}{t}}}{\Psi(\varepsilon)}\Big)^{Nt}.
\end{align*}
Thus,
\begin{align*}
\frac{\log\operatorname{span}(\Xi,N,d,\Gamma,\varepsilon)}{-N\log\varepsilon}\leq s-t+t\frac{\log\Psi(\varepsilon)}{\log\varepsilon}
\end{align*}
and
\begin{align*}
\lim_{\varepsilon\to 0}\limsup_{N\to\infty}\frac{\log\operatorname{span}(\Xi,N,d,\Gamma,\varepsilon)}{-N\log\varepsilon}\leq s.
\end{align*}
This contradicts $s<\overline{\mdim}^{M}(\Xi,d,\Gamma)$, then $\overline{\mdim}_{\Psi}(\Xi,d,\Gamma)\geq\overline{\mdim}^{M}(\Xi,d,\Gamma)$. Finally, by using Lemma \ref{l3}, we have
$$
\overline{\mdim}_{\Psi}(\Xi,d,\Gamma)=\overline{\mdim}^{M}(\Xi,d,\Gamma).
$$
\end{proof}

\subsection{A mass distribution principle}

Similar to other dimension concepts, fundamental techniques are valuable for studying and calculating mean intermediate dimensions in specific scenarios. The mass distribution principle is commonly employed to determine lower bounds for the Hausdorff dimension by analyzing the local behavior of measures supported on the set. Here, we present the natural analogues for $\underline{\mdim}_{\Psi}$ and $\overline{\mdim}_{\Psi}$, which are derived through a straightforward modification of the standard proofs for intermediate and Hausdorff dimensions.

\begin{lemma}\label{l321}
Let $(\Omega,\Gamma)$ be a TDS, $a, c, >0$ 
and $s\geq 0$.
\begin{enumerate}
\item[(A)] If there exist a positive decreasing sequence $\mathcal{E}=\{\varepsilon_{n}\}_{n}$ with $\lim_{n\to \infty}\varepsilon_{n}=0$ and a positive increasing sequence $\mathcal{N}=\{N_{m}\}_{m}$ with $\lim_{m\to \infty}N_{m}=\infty$ such that for each $n,m \in \mathbb{N}$ there exists a Borel measure $\nu_{n,m}$ with
\begin{enumerate}
    \item[(a)] $\nu_{n,m}\left(\operatorname{supp}\left(\nu_{n,m}\right)\right) \geq a$;
    \item[(b)] For every Borel subset $U \subseteq \Omega$ with $\Psi(\varepsilon_{n}) < \diam_{d_{N_m}}(U) \leq \varepsilon_{n}$ we have $\nu_{n,m}(U) \leq c\diam_{d_{N_m}}(U)^{N_{m} s}$.
\end{enumerate}
Then $\overline{\mdim}_{\Psi} (\Omega,d,\Gamma) \geq s$.

\item[(B)] If, moreover, for all sufficiently small $\varepsilon >0$ and sufficiently large $N>0$, there exists a Borel measure $\nu_{\varepsilon,N}$ with
\begin{enumerate}
    \item[(a)] $\nu_{\varepsilon,N}\left(\operatorname{supp}\left(\nu_{\varepsilon,N}\right)\right) \geq a$;
    \item[(b)] For every Borel subset $U \subseteq \Omega$ with $\Psi(\varepsilon) < \diam_{d_{N}}(U) \leq \varepsilon$ we have $\nu_{\varepsilon,N}(U) \leq c\diam_{d_{N}}(U)^{N s}$.
\end{enumerate}
Then $\underline{\mdim}_{\Psi} (\Omega,d,\Gamma) \geq s$.
\end{enumerate}
\end{lemma}
\begin{proof}
We only prove $(A)$, because the proof for $(B)$ follows a similar approach to the proof for $(A)$. Let $n,m \in \mathbb{N}$ and $\{\mathscr{B}_{d_{N_m}}(\widetilde{\mathbf{u}}_i,r_i)\}_{i}$ be a cover of $\Omega$ such that $\Psi\left(\varepsilon_{n}\right) < 2r_i \leq \varepsilon_{n}$ for all $i$, then

\begin{equation}\label{e51}
a \leq \nu_{n,m}\left(\operatorname{supp}\left(\nu_{n,m}\right)\right)\leq\nu_{n,m}\left(\bigcup_{i}\mathscr{B}_{d_{N_m}}(\widetilde{\mathbf{u}}_i,r_i) \right) \leq \sum_{i} \nu_{n,m}\left(\mathscr{B}_{d_{N_m}}(\widetilde{\mathbf{u}}_i,r_i)\right) \leq c \sum_{i}(2r_i)^{N_m s}.
\end{equation}
Therefore $\displaystyle\sum_{i\in\mathbb{N}}(2r_i)^{N_m s} \geq\frac{c}{a}>0$, so $\overline{\mdim}_{\Psi} (\Omega,d,\Gamma) \geq s$.

\end{proof}

\subsection{A Frostman type lemma}

The Frostman lemma, a powerful tool in fractal geometry and a counterpart to the mass distribution principle, is also applicable here. We present a version tailored for mean intermediate dimensions.

\begin{proposition}\cite[Theorem 2.2]{28}\label{p1}
Let $N\in\mathbb{N}$. For the metric space $(\Omega,d_N)$, assume that constants $A_0 >0$, $0<c_0 \leq C_0 <+\infty$ and $0<\varepsilon< 1$ satisfy
$$
12A_{0}^{3}C_{0}\varepsilon\leq c_0.
$$
If there is a set of points $\mathcal{S}_{k,N}=\{\widetilde{\mathbf{u}}_{k,\alpha,N}\}_{\alpha}\subseteq \Omega$ for all $k\in\mathbb{N}$ such that $\{\widetilde{\mathbf{u}}_{k,\alpha,N}\}_{\alpha}$ is a $c_0 \varepsilon^{k}$-separated subset of $\Omega$ and $\{\mathscr{B}_{d_N}(\widetilde{\mathbf{u}}_{k,\alpha,N},C_0 \varepsilon^{k})\}$ is a cover of $\Omega$ for all $k\in \mathbb{N}$, then there exist a set $\mathbb{Q}_{k,N}=\{\mathscr{Q}_{k,\alpha,N}\}_{\alpha}$ of subsets of $\Omega$ such that:
\begin{enumerate}
  \item[(a)]  $\Omega=\bigcup_{\alpha} \mathscr{Q}_{k,\alpha,N}$ with the union disjoint;

  \item[(b)]Let $c_1 =\frac{c_0}{3A_0^2}$ and $C_1 =2A_0 C_0$, then there is a point $\widetilde{\mathbf{u}}_{k,\alpha,N}\in\mathcal{S}_{k,N}$ such that $$ U_{d_N}\left(\widetilde{\mathbf{u}}_{k,\alpha,N}, c_{1}\varepsilon^{k}\right) \subseteq \mathscr{Q}_{k,\alpha,N} \subseteq \mathscr{B}_{d_N}\left(\widetilde{\mathbf{u}}_{k,\alpha,N}, C_{1}\varepsilon^{k}\right),$$
  and call the $\mathscr{Q}_{k,\alpha,N}$ is a cube with centre $\widetilde{\mathbf{u}}_{k,\alpha,N}$.

  \item[(c)] if $k_1, k_2 \in \mathbb{N}$ with $k_1 \leq k_2$ then for all $\alpha$ and $\beta$, either $\mathscr{Q}_{k_1,\alpha,N} \cap \mathscr{Q}_{k_2,\beta,N}=\varnothing$ or $\mathscr{Q}_{k_2,\beta,N}\subseteq \mathscr{Q}_{k_1,\alpha,N}$. In this case, we say $\mathscr{Q}_{k_1,\alpha,N}$ is a parent of $\mathscr{Q}_{k_2,\beta,N}$.
\end{enumerate}
\end{proposition}

\begin{lemma}\label{l52}
Let $(\Omega,\Gamma)$ be a TDS which is dynamically uniformly perfect and $\mdim^{A}(\Omega,d,\Gamma)<\infty$.
\begin{enumerate}
   \item[(A)] If $\overline{\mdim}_{\Psi}(\Omega,d,\Gamma)>0$ then for all $s\in (0,\overline{\mdim}_{\Psi}(\Omega,d,\Gamma))$ there exists a constant $c\in (0,\infty)$ which satisfies for sufficiently small $\varepsilon_{0}>0$ and sufficient large $N_0 >0$, there exist a decreasing sequence $\mathcal{E}=\{\varepsilon_\alpha\}_{\alpha}$ which satisfies $\lim_{\alpha\to\infty}\varepsilon_{\alpha}=0$ and $0<\varepsilon_\alpha <\varepsilon_0$ for all $\alpha$,  and an increasing sequence $\mathcal{N}=\{N_\beta\}_{\beta}\subseteq \mathbb{N}$ which satisfies $\lim_{\beta\to\infty}\varepsilon_{\beta}=\infty$ and $N_\beta >N_0$ for all $\beta$ such that if $\varepsilon\in\mathcal{E}$ and $N\in \mathcal{N}$, there is a Borel probability measure $\nu_{\varepsilon,N}$ with finite support $\operatorname{supp}\left(\nu_{\varepsilon^,N}\right)$ for all $\widetilde{\mathbf{u}} \in \Omega$ and $\Psi\left(\varepsilon\right) < r \leq \varepsilon$, then
   $$
   \nu_{\varepsilon,N}(\mathscr{B}_N (\widetilde{\mathbf{u}}, r)) \leq c r^{Ns}.
   $$
   \item[(B)] If $\underline{\mdim}_{\Psi}(\Omega,d,\Gamma)>0$ then for all $s \in\left(0, \underline{\mdim}_{\Psi}(\Omega,d,\Gamma)\right)$ there exists $c \in(0, \infty)$ such that for all sufficiently small $\varepsilon$ and sufficient large $N>0$ there exists a Borel probability measure $\nu_{\varepsilon,N}$ with finite support $\operatorname{supp}\left(\nu_{\varepsilon,N}\right)$ such that if $\widetilde{\mathbf{u}} \in \Omega$ and $\Psi\left(\varepsilon\right) < r \leq \varepsilon$, then
   $$
   \nu_{\varepsilon,N}(\mathscr{B}_N (\widetilde{\mathbf{u}}, r)) \leq c r^{Ns}.
   $$
\end{enumerate}
\end{lemma}
\begin{proof}
We only prove $(B)$, because the proof for $(A)$ follows a path similar to that for $(B)$.

Set $\mdim^{A}(\Omega,d,\Gamma)=a<\infty$, then there exist $C>0$ and $N(C)\in\mathbb{N}$ such that for all $N>N(C)$, we have
\begin{align*}
    \operatorname{span}(\Omega\cap \mathscr{B}_{d_N}(\widetilde{\mathbf{u}},R),N,d,\Gamma,r)\leq C{\Big(\frac{R}{r}\Big)}^{a},
\end{align*}
where $0<r<R$ and $\widetilde{\mathbf{u}}\in \Xi$.

Since $\Omega$ is dynamically uniformly perfect, set $0< c_2 <1$ such that $\Omega$ is $c_2$-dynamically uniformly perfect. Because $\lim_{\widetilde{\mathbf{u}}\to 0}\frac{\Psi(\widetilde{\mathbf{u}})}{\widetilde{\mathbf{u}}}=0$, there is $0<\varepsilon(c_2) <1$ such that for all $0<\varepsilon<\varepsilon(c_2)$, we have $\frac{\Psi(\varepsilon)}{\varepsilon}<\frac{c_2}{320}$.

By $s<\underline{\mdim}_{\Psi}(\Omega,d,\Gamma)$, there are $0<\delta<1$ and $0<\varepsilon(\delta)<1$, for any $0<\varepsilon<\varepsilon(\delta)$ there exists $N(\delta,\varepsilon)\in\mathbb{N}$, such that for all $N>N(\delta,\varepsilon)$, any cover $\{\mathscr{B}_{d_N}(\widetilde{\mathbf{u}}_i,r_i)\}_{i\in\mathbb{N}}$ of $\Omega$ which satisfies $\Psi(\varepsilon)< 2 r_i \leq \varepsilon$ and 
\begin{equation}\label{ne3}
\sum_{i\in\mathbb{N}}(2r_i)^{Ns}\geq \delta.
\end{equation}
For this $\delta$, let $0<\varepsilon<\min\{\varepsilon(c_2),\varepsilon(\delta)\}$ and $N>\max\{N(\delta,\varepsilon),N(C)\}$.

Let $R=\diam_{d}(\Omega)$, by $\Gamma(\Omega)\subseteq \Omega$, we have $\diam_{d_N}(\Omega)\leq \diam_{d}(\Omega)=R$ for every $N\in\mathbb{N}$. Set $\gamma=\min\{\frac{1}{20},R\}$. For each $k\in\mathbb{N}$, clearly $\gamma^{k}\leq R$ and
\begin{align*}
    \operatorname{span}(\Omega\cap \mathscr{B}_{d_N}(\widetilde{\mathbf{u}},R),N,d,\Gamma,\gamma^{k})=\operatorname{span}(\Omega,N,d,\Gamma,\gamma^{k})\leq C{\Big(\frac{R}{\gamma^{k}}\Big)}^{a} <+\infty.
\end{align*}
By the finite Vitali covering lemma, there exists a finite $\gamma^{k}$-separated subset $\mathcal{S}_{k,N}=\{\widetilde{\mathbf{u}}_{k,\alpha,N}\}_{\alpha}$ of $\Omega$. Then applying Proposition \ref{p1} with $c_0=C_0=1$, $A_0 =1$, $c_1=\frac{1}{3}$ and $C_1=2$, for each $k\in\mathbb{N}$, there exists a set $\mathbb{Q}_{k,N}=\{\mathscr{Q}_{k,\alpha,N}\}_{\alpha}$ of subsets of $\Omega$ such that:
\begin{enumerate}
  \item[(a)]  $\Omega=\bigcup_{\alpha} \mathscr{Q}_{k,\alpha,N}$ with the union disjoint,

  \item[(b)] There is a point $\widetilde{\mathbf{u}}_{k,\alpha,N}\in\mathcal{S}_{k,N}$ such that $$U_{d_N}\left(\widetilde{\mathbf{u}}_{k,\alpha,N},\frac{\gamma^k}{4}\right) \subseteq U_{d_N}\left(\widetilde{\mathbf{u}}_{k,\alpha,N}, c_{1}\gamma^{k}\right) \subseteq \mathscr{Q}_{k,\alpha,N} \subseteq \mathscr{B}_{d_N}\left(\widetilde{\mathbf{u}}_{k,\alpha,N}, C_{1}\gamma^{k}\right)= \mathscr{B}_{d_N}\left(\widetilde{\mathbf{u}}_{k,\alpha,N}, 2\gamma^{k}\right),$$

  \item[(c)] if $k_1, k_2 \in \mathbb{N}$ with $k_1 \leq k_2$ then for all $\alpha$ and $\beta$, either $\mathscr{Q}_{k_1,\alpha,N} \cap \mathscr{Q}_{k_2,\beta,N}=\varnothing$ or $\mathscr{Q}_{k_2,\beta,N}\subseteq \mathscr{Q}_{k_1,\alpha,N}$.
\end{enumerate}
Note that the $\mathscr{Q}_{k,\alpha,N}$ is a cube with centre $\widetilde{\mathbf{u}}_{k,\alpha,N}$. 

Define $m=m(\varepsilon)$ be the largest integer with $\frac{1}{2}\gamma^{m}<\Psi(\varepsilon)\leq \frac{1}{2}\gamma^{m-1}$. Since $\Psi(\varepsilon)<\frac{c_2}{320}\varepsilon<\frac{1}{320}\varepsilon$, we have $160\gamma^{m}<\varepsilon$. Let $l$ be the largest integer such that $8\gamma^{m-l}\leq \varepsilon$, it is clear that $l\geq 1$(recall that $\gamma\leq\frac{1}{20}$). According to the condition $(b)$,  we have $\diam_{d_N}(\mathscr{Q})\leq\frac{\varepsilon}{2}$ for all $\mathscr{Q}\in \mathbb{Q}_{m-l,N}$.

Now define the measure
$\nu_m$ on $\Omega$ as follows: for each $\mathscr{Q}\in\mathbb{Q}_{m,N}$ and Borel set $B\subseteq \Omega$, let
$$
\nu_{m,N} (B) = \sum_{\mathscr{Q}\in\mathbb{Q}_{m,N}: B\cap \mathscr{Q}\neq \emptyset} \gamma^{ms} M_{\widetilde{\mathbf{u}}_\mathscr{Q}},
$$
where $\widetilde{\mathbf{u}}_\mathscr{Q}$ is the centre of $\mathscr{Q}$ and $M_{\widetilde{\mathbf{u}}_\mathscr{Q}}$ is a unit point mass at $\widetilde{\mathbf{u}}_\mathscr{Q}$. According to $\nu_{m,N}$, define $\nu_{m-1,N}$ as follows,
\begin{align*}
\nu_{m-1,N} |_{\mathscr{Q}}= \min\left\{1, \gamma^{(m-1)s}\nu_{m,N}(\mathscr{Q})^{-1}\right\}\nu_{m,N} |_{\mathscr{Q}},
\end{align*}
where $\mathscr{Q}\in\mathbb{Q}_{m-1,N}$. Continuing inductively, $\nu_{m-k-1,N}$ is obtained from $\nu_{m-k,N}$ by
\begin{align*}
\nu_{m-k-1,N} |_{\mathscr{Q}}= \min\left\{1, \gamma^{(m-k-1)s}{\nu_{m-k,N}(\mathscr{Q})}^{-1}\right\}\nu_{m-k,N} |_{\mathscr{Q}},
\end{align*}
where $\mathscr{Q}\in\mathbb{Q}_{m-k-1,N}$ and $k=\{0,1,\cdots,l\}$. Assume that there are $k\in\{0,1,\cdots,l\}$ and $\mathscr{Q}_{m-k}\in\mathbb{Q}_{m-k,N}$ such that $\nu_{m-l,N}(\mathscr{Q}_{m-k})> \gamma^{(m-k)s}$. By construction, we have 
\begin{align*}
\gamma^{(m-k)s}<\nu_{m-l,N}(\mathscr{Q}_{m-k})\leq \nu_{m-l+1,N}(\mathscr{Q}_{m-k})\leq \cdots\leq \nu_{m-k,N}(\mathscr{Q}_{m-k})\leq \nu_{m-k+1,N}(\mathscr{Q}_{m-k})
\end{align*}
and
\begin{align*}
\nu_{m-k,N}(\mathscr{Q}_{m-k})=\min\left\{1,\frac{\gamma^{(m-k)s}}{\nu_{m-k+1,N}(\mathscr{Q}_{m-k})}\right\}\nu_{m-k+1,N}(\mathscr{Q}_{m-k})=\gamma^{(m-k)s},
\end{align*}
which contradicts the hypothesis.
Then if $k\in\{0,1,,\cdots,l\}$ and $\mathscr{Q}_{m-k}\in\mathbb{Q}_{m-k,N}$, we have
\begin{equation}\label{ne1}
\nu_{m-l,N}(\mathscr{Q}_{m-k})\leq \gamma^{(m-k)s}.
\end{equation}
By the condition $(b)$, \eqref{ne1} and $0<c_2 <1$,
\begin{equation}\label{ee2}
 \nu_{m-l,N}(\mathscr{Q}_{m-k})\leq \gamma^{(m-k)s}\leq 4^{s}{\diam_{d_N}(\mathscr{Q}_{m-k})}^{s}\leq 4^{s}c_{2}^{-s}{\diam_{d_N}(\mathscr{Q}_{m-k})}^{s}.
\end{equation}
Obviously, $\nu_{m,N}(\mathscr{Q}_{m})=\gamma^{ms}$ for any $\mathscr{Q}_{m}\in\mathbb{Q}_{m,N}$.
If $k \in$ $\{0,1, \ldots, l-1\}$ and $\mathscr{Q}_{m-k} \in \mathbb{Q}_{m-k,N}$ satisfies $\nu_{m-k,N}\left(\mathscr{Q}_{m-k}\right)=\gamma^{(m-k) s}$ and $\mathscr{Q}_{m-k-1} \in \mathbb{Q}_{m-k-1,N}$ is the parent of $\mathscr{Q}_{m-k}$, therefore, we have 
\begin{enumerate}
\item[(d)] if $\gamma^{m-k-1}\geq \nu_{m-k,N}(\mathscr{Q}_{m-k-1})$, then
\begin{align*}
\nu_{m-k-1,N}(\mathscr{Q}_{m-k})=\min\left\{ 1,\frac{\gamma^{(m-k-1)s}}{\nu_{m-k,N}(\mathscr{Q}_{m-k-1})} \right\}\nu_{m-k,N}(\mathscr{Q}_{m-k})=\gamma^{(m-k)s},
\end{align*}
\item[(e)] and if $\gamma^{m-k-1}< \nu_{m-k,N}(\mathscr{Q}_{m-k-1})$, then
\begin{align*}
\nu_{m-k-1,N}(\mathscr{Q}_{m-k-1})=\min\left\{ 1,\frac{\gamma^{(m-k-1)s}}{\nu_{m-k,N}(\mathscr{Q}_{m-k-1})} \right\}\nu_{m-k,N}(\mathscr{Q}_{m-k-1})=\gamma^{(m-k-1)s},
\end{align*}
\end{enumerate}
i.e. by the construction of $\nu_{m-k-1}$, either 
$$
\nu_{m-k-1,N}\left(\mathscr{Q}_{m-k}\right)=\gamma^{(m-k) s}\quad \text{or}\quad \nu_{m-k-1,N}\left(\mathscr{Q}_{m-k-1}\right)=\gamma^{(m-k-1) s}.
$$ 
Therefore for all $\widetilde{\mathbf{u}} \in \Omega$ there is at least one $k \in\{0,1, \ldots, l\}$ and $\mathscr{Q}_{\widetilde{\mathbf{u}}} \in \mathbb{Q}_{m-k,N}$ with $\widetilde{\mathbf{u}} \in \mathscr{Q}_{\widetilde{\mathbf{u}}}$ such that
\begin{equation}\label{ee3}
\nu_{m-l,N}\left(\mathscr{Q}_{\widetilde{\mathbf{u}}}\right)=\gamma^{(m-k) s} \geq 4^{-s}{\diam_{d_N}(\mathscr{Q}_{\widetilde{\mathbf{u}}})}^{s},
\end{equation}
where the inequality is by the condition $(b)$.

For each $\widetilde{\mathbf{u}} \in \Omega$, choosing $\mathscr{Q}_{\widetilde{\mathbf{u}}}$ such that \eqref{ee3} is satisfied and moreover $\mathscr{Q}_{\widetilde{\mathbf{u}}} \in \mathbb{Q}_{m-k,N}$ for the largest possible $k \in\{0,1, \ldots, l\}$ yields a finite collection of cubes $\left\{\mathscr{Q}_{i}\right\}_i$ which covers $\Omega$, this is because $\Omega$ is compact. For each $i$, let $\widetilde{\mathbf{u}}_{i}$ be the centre of $\mathscr{Q}_{i}$, and by the dynamically uniformly perfect condition there exists $p_{i} \in \Omega$ such that 
\begin{equation}\label{ne2}
\Psi\left(\varepsilon\right) < d_{N}\left(p_{i}, \widetilde{\mathbf{u}}_{i}\right) \leq \Psi\left(\varepsilon\right) / c_{2} \leq \varepsilon / 2.
\end{equation}
Letting $U_{i}=\mathscr{Q}_{i} \cup\left\{p_{i}\right\}$, by the condition $(b)$ and \eqref{ne2}, we have $\Psi\left(\varepsilon\right) < \diam_{d_N}(U_{i}) \leq \varepsilon$. Then $\left\{U_{i}\right\}_i$ covers $\Omega$. For each $i$
\begin{equation}\label{ne4}
\diam_{d_N}(U_i) \leq\diam_{d_N}(\mathscr{Q}_i)+\Psi\left(\varepsilon\right) / c_{2} \leq\left(1+1 / c_{2}\right)\diam_{d_N}(\mathscr{Q}_i),
\end{equation}
and there is a $\widetilde{\mathbf{u}}^{c}_{i}\in \Omega$ such that $U_i\subseteq \mathscr{B}_{d_N}(\widetilde{\mathbf{u}}^{c}_{i},\diam_{d_N}(U_i))$. It is clear that $\diam_{d_N}(U_i)<1$ for every $i$ and $\{\mathscr{B}_{d_N}(\widetilde{\mathbf{u}}^{c}_{i},\diam_{d_N}(U_i))\}_{i}$ is a cover of $\Omega$. It follows from \eqref{ne3} that
\begin{equation}\label{ne5}
\sum_{i}(2\diam_{d_N}(U_i))^{s}\geq\sum_{i}(2\diam_{d_N}(U_i))^{Ns}\geq \delta.
\end{equation}
By using \eqref{ee3}, \eqref{ne4} and \eqref{ne5}, we obtain
\begin{align*}
\nu_{m-l,N}(\Omega) & =\sum_{i} \nu_{m-l,N}\left(\mathscr{Q}_{i}\right) \\&\geq \sum_{i} 8^{-s}{(2\diam_{d_N}(\mathscr{Q}_i))}^{s} \\
& \geq 8^{-s}\left(1+1 / c_{2}\right)^{-s} \sum_{i}{(2\diam_{d_N}(U_i))}^{s}\\& \geq 8^{-s}\left(1+1 / c_{2}\right)^{-s} \delta.
\end{align*}
Define $\nu_{\varepsilon,N}:=\left(\nu_{m-l,N}(\Omega)\right)^{-1} \nu_{m-l,N}$, which is clearly a Borel probability measure with finite support.

Since $\mdim_{A}(\Omega,d,\Gamma)=a<+\infty$, for every $p\in \Omega$ and $w>0$,
 \begin{align*}
\operatorname{span}(\Omega\cap \mathscr{B}_{d_N}(p,13w),N,d,\Gamma,w)=\operatorname{span}(\mathscr{B}_{d_N}(p,13w),N,d,\Gamma,w)\leq C{\Big(\frac{13w}{w}\Big)}^{a}=13^{a}C.
\end{align*}
Let $\Psi(\varepsilon)< r\leq \varepsilon$ and $j=j(r)$ be the integer in $\{0,1, \ldots, l\}$ such that $\gamma^{m-j+1}<r\leq \gamma^{m-j}$, such an integer exists by the definition of $m$. Let $\widetilde{\mathbf{u}}\in \Omega$, suppose $\mathscr{B}_{d_N}(\widetilde{\mathbf{u}}, r) \cap \mathscr{Q}_{m-j} \neq \emptyset$ for some $\mathscr{Q}_{m-j} \in \mathbb{Q}_{m-j,N}$ with centre $z_{m-j}$. Then there exists $z \in \mathscr{B}_{d_N}(\widetilde{\mathbf{u}}, r) \cap \mathscr{Q}_{m-j}$, and by the condition $(b)$ and the definition of $j$,
$$
d_{N}\left(\widetilde{\mathbf{u}}, z_{m-j}\right) \leq d_{N}\left(\widetilde{\mathbf{u}}, z\right)+d_{N}\left(z, z_{m-j}\right) \leq 2 r+2\gamma^{m-j} \leq 6\gamma^{m-j} .
$$
Therefore $z_{m-j} \in \mathscr{B}_{d_N}\left(\widetilde{\mathbf{u}}, 6\gamma^{m-j}\right)$, and the centres of the cubes in $\mathbb{Q}_{m-j,N}$ which intersect $\mathscr{B}_{d_N}(\widetilde{\mathbf{u}}, r)$ form a $\gamma^{m-j}$-separated subset of $\mathscr{B}_{d_N}\left(\widetilde{\mathbf{u}}, 6\gamma^{m-j}\right)$. But

$$
\operatorname{span}(\mathscr{B}_{d_N}(\widetilde{\mathbf{u}},6\gamma^{m-j}),N,d,\Gamma,\frac{6\gamma^{m-j}}{13})\leq 13^{a}C.
$$
Therefore there are at most $13^{a}C$ such centres, so at most $13^{a}C$ elements of $\mathbb{Q}_{m-j,N}$ which intersect $\mathscr{B}_{d_N}(\widetilde{\mathbf{u}}, r)$. Therefore by the definition of $j$,

$\nu_{\varepsilon,N}\left(\mathscr{B}_{d_N}(\widetilde{\mathbf{u}}, r)\right)=\left(\nu_{m-l,N}(\Omega)\right)^{-1} \nu_{m-l,N}\left(\mathscr{B}_{d_N}(\widetilde{\mathbf{u}}, r)\right) \leq 13^{a}C\left(\nu_{m-l,N}(\Omega)\right)^{-1} \gamma^{(m-j) s} \leq c r^{s}$,

where $c:=13^{a}C 8^{s}\left(1+1 / c_{2}\right)^{s} \delta^{-1}\gamma^{-s}$, as required.

\end{proof}

Now we give the equivalent definition of mean $\Psi$-intermediate dimension.
\begin{proposition}\label{p343}
Let $(\Omega,\Gamma)$ be a TDS which satisfies dynamically uniformly perfect, $\Psi$ be an admissible function and $\mdim^{A}(\Omega,d,\Gamma)< +\infty$. 
\begin{enumerate}
\item[(a)] If $\overline{\mdim}_{\Psi}(\Omega,d,\Gamma)>0$, then
\begin{align*}
\overline{\mdim}_{\Psi}(\Omega,d,\Gamma)=\sup\Bigg\{&s>0 \mid \text{ there exists } C\in (0,+\infty) \text{ such that for all } 0<\varepsilon_0 <1 \text{ and } N_0 \in\mathbb{N}, \\
&\text{ there exist } 0<\varepsilon\leq \varepsilon_0, N>N_0 \text{ and a measure } \nu_{\varepsilon,N}  
\text{ such that if } U \text{ is a Borel }\\
&\text{ subset of } \Omega \text{ which satisfies } \Psi(\varepsilon)<\diam_{d_N}(U)\leq \varepsilon
\text{ then } \\&\nu_{\varepsilon,N}(U)\leq C \diam_{d_N}(U)^{Ns}\Bigg\}.
\end{align*}
\item[(b)] If $\underline{\mdim}_{\Psi}(\Omega,d,\Gamma)>0$, then
\begin{align*}
\underline{\mdim}_{\Psi}(\Omega,d,\Gamma)=\sup\Bigg\{&s>0 \mid \text{ there exists } C\in (0,+\infty) \text{ such that there exists  } 0<\varepsilon(C) <1 \text{ for any }\\
&0<\varepsilon<\varepsilon(C)\text{ there exists } N(C,\varepsilon) \in\mathbb{N} \text{ which satisfies for any }N>N(C,\varepsilon) ,\\ &\text{ there exists }
\text{ a measure }\nu_{\varepsilon,N} 
\text{ such that if } U \text{ is a Borel subset of } \Omega \text{ with}\\  & \Psi(\varepsilon)<\diam_{d_N}(U)\leq \varepsilon
\text{ then } \nu_{\varepsilon,N}(U)\leq C \diam_{d_N}(U)^{Ns}\Bigg\}.
\end{align*}
\end{enumerate}
\end{proposition}
\begin{proof}
We focus solely on proving $(b)$ since the proof of $(a)$ follows a similar approach. We denote by $low$ the supremum on the right-hand side of the equation $(b)$.

Set $s\in (0,\underline{\mdim}_{\Psi}(\Omega,d,\Gamma))$. Then by using Lemma \ref{l52}, for all sufficiently small $\varepsilon$ and sufficiently large $N$, there is a Borel probability measure $\nu_{\varepsilon,N}$ which has finite support and satisfies if $\widetilde{\mathbf{u}}\in \Omega$ and $\Psi(\varepsilon)< r\leq\varepsilon$ then $\nu_{\varepsilon,N}(\mathscr{B}_{d_N}(\widetilde{\mathbf{u}},r))\leq cr^{Ns}$. 

Let $U$ be a Borel subset of $\Omega$ satisfying $\Psi(\varepsilon)<\diam_{d_N}(U)\leq \varepsilon$. Since $\mdim^{A}(\Omega,d,\Gamma)=a< +\infty$, there are $K>0$ and $N(K)>0$ such that for all $N>N(K)$,
$$
\operatorname{span}(\mathscr{B}_{d_N}(\widetilde{\mathbf{u}},2\diam_{d_N}(U)),N,d,\Gamma,\diam_{d_N}(U))\leq 2^a K
$$
holds for all $\widetilde{\mathbf{u}}\in \Omega$.

Set $M=2^a K$ and $C=Mc$. 
If $U\cap \supp(\nu_{\varepsilon,N}) \neq\emptyset$, take $\widetilde{\mathbf{u}}\in U\cap \supp(\nu_{\varepsilon,N})$. Then 
$$
U\cap \supp(\nu_{\varepsilon,N})\subseteq \mathscr{B}_{d_N}(\widetilde{\mathbf{u}},2\diam_{d_N}(U)),
$$
so there exist $\widetilde{\mathbf{u}}_1,\cdots,\widetilde{\mathbf{u}}_M\in \Omega$ such that 
$$
U\cap \supp(\nu_{\varepsilon,N})\subseteq \mathscr{B}_{d_N}(\widetilde{\mathbf{u}},2\diam_{d_N}(U))\subseteq \bigcup_{i=1}^{M}\mathscr{B}_{d_N}(\widetilde{\mathbf{u}}_i,\diam_{d_N}(U)).
$$
Therefore, we obtain
\begin{align*}
\nu_{\varepsilon,N}(U)\leq \sum_{i=1}^{M}\nu_{\varepsilon,N}(\mathscr{B}_{d_N}(\widetilde{\mathbf{u}}_i,\diam_{d_N}(U)))\leq Mc\diam_{d_N}(U)^{Ns}=C\diam_{d_N}(U)^{Ns}.
\end{align*}
And if $U\cap \supp(\nu_{\varepsilon,N})=\emptyset$ then $\nu_{\varepsilon,N}=0\leq C\diam_{d_N}(U)^{Ns}$. Thus $s<low$ and $\underline{\mdim}_{\Psi}(\Omega,d,\Gamma)\leq low$. 

For the reverse inequality, let $t\in (0,low)$ then it follows from Lemma \ref{l321} that $t\leq \underline{\mdim}_{\Psi}(\Omega,d,\Gamma)$ and $low\leq \underline{\mdim}_{\Psi}(\Omega,d,\Gamma)$.
\end{proof}

\subsection{Hölder distortion}
In dimension theory, a significant aspect lies in its practical applications. The concept revolves around the idea that meaningful interpolation between two dimensions leads to more insightful information compared to considering the dimensions independently. This enhanced information, in turn, facilitates more robust applications, particularly in deriving new insights into mean Hausdorff dimensions and metric mean dimensions using mean intermediate dimensions. One notable application of dimension theory arises from the observation that dimensions often remain invariant, or approximately invariant in a quantifiable manner, under certain transformations. For instance, the Hausdorff, box,  Assouad, and intermediate dimensions maintain invariance under bi-Lipschitz maps, serving as valuable invariants for classification up to bi-Lipschitz image transformations. Furthermore, the Assouad spectrum and intermediate dimensions also exhibit invariance under bi-Lipschitz maps, offering a continuum of invariants within the same classification context. Our exploration now delves into the behavior of mean intermediate dimensions under Hölder and Lipschitz maps.

We say that a map $f: \Omega \rightarrow \Delta$ is Hölder or $(C,\alpha)$-Hölder if there exist constants $\alpha\in(0,1]$ and $C\in[0, \infty)$ such that
$$
d^{\Delta}\left(f\left(\widetilde{\mathbf{u}}_{1}\right), f\left(\widetilde{\mathbf{u}}_{2}\right)\right) \leq C d^{\Omega}\left(\widetilde{\mathbf{u}}_{1}, \widetilde{\mathbf{u}}_{2}\right)^{\alpha}  
$$
holds for all $\widetilde{\mathbf{u}}_{1}, \widetilde{\mathbf{u}}_{2} \in \Omega$, where $d^\Omega$ and $d^\Delta$ are metrics on $\Omega$ and $\Delta$ respectively.
\begin{proposition}\label{p341}
Let $\Psi$, $\Psi_{1}$ be admissible functions, $(\Omega,\Gamma_\Omega)$, $(\Delta,\Gamma_\Delta)$ be two TDSs, and $d^\Omega$, $d^\Delta$ be metrics on $\Omega$ and $\Delta$ respectively. If $f: \widetilde{\Xi} \rightarrow \Delta$ is a $(C,\alpha)$-Hölder map for some $\widetilde{\Xi} \subseteq \Omega$ which satisfies 
\begin{enumerate}
\item[(a)] $(\Delta,d^\Delta)$ is $p$-dynamically uniformly perfect, 
\item[(b)] $\Psi_{1}(\varepsilon)\leq Cp\Psi\left((\frac{\varepsilon}{4C})^{\frac{1}{\alpha}}\right)^{\alpha}$ for all sufficiently small $\varepsilon$,
\item[(c)] $C\leq\frac{1}{4}$,
\end{enumerate}
then
$$
\overline{\mdim}_{\Psi_{1}} (f(\widetilde{\Xi}),d^\Delta,\Gamma_\Delta) \leq \frac{1}{\alpha}\overline{\mdim}_{\Psi} (\widetilde{\Xi},d^\Omega,\Gamma_\Omega)
$$
and
$$
\underline{\mdim}_{\Psi_{1}} (f(\widetilde{\Xi}),d^\Delta,\Gamma_\Delta) \leq \frac{1}{\alpha}\underline{\mdim}_{\Psi} (\widetilde{\Xi},d^\Omega,\Gamma_\Omega).
$$
\end{proposition}
\begin{proof}
We only prove the first inequality, as the proof of the second inequality is similar. 
    
Let $t>\frac{1}{\alpha}\overline{\mdim}_{\Psi} (\widetilde{\Xi},d^\Delta,\Gamma_\Delta)$, then there exists $s>\overline{\mdim}_{\Psi} (\widetilde{\Xi},d^\Omega,\Gamma_\Omega)$ such that $t>\frac{s}{\alpha}$. Let $0<\delta<1$, then there exists $0<\varepsilon(\delta)<1$ such that for any $0<\varepsilon<\varepsilon(\delta)$, there exists $N(\delta,\varepsilon)\in\mathbb{N}$ such that for all $N>N(\delta,\varepsilon)$, there is a cover $\{\mathscr{B}_{d_N^\Omega}(\widetilde{\mathbf{u}}_i,r_i)\}_{i\in\mathbb{N}}$ of $\widetilde{\Xi}$ with $\Psi(\varepsilon) < 2r_i \leq \varepsilon$ for all $i\in\mathbb{N}$, and
\begin{equation}\label{e43}
\sum_{i\in\mathbb{N}}(2r_i)^{Ns} <\delta .
\end{equation}
Then $\{f(\mathscr{B}_{d_N^\Omega}(\widetilde{\mathbf{u}}_i,r_i))\}_{i\in\mathbb{N}}$ covers $f(\widetilde{\Xi})$. By condition $(a)$, $p$ is the constant such that $(\Delta,d^\Delta)$ is dynamically uniformly perfect, then for every $i\in\mathbb{N}$ and $\widetilde{\mathbf{v}}_i\in f(\mathscr{B}_{d_N^\Omega}(\widetilde{\mathbf{u}}_i,r_i))$, there exists $\widetilde{\mathbf{v}}_i^*\in \Delta$ such that $$\Psi_{1}(4C\varepsilon^{\alpha})\leq d^{\Delta}_{N}(\widetilde{\mathbf{v}}_i,\widetilde{\mathbf{v}}_i^*)\leq \frac{\Psi_{1}(4C\varepsilon^{\alpha})}{p}\leq C\Psi(\varepsilon)^{\alpha}\leq C(2r_i)^{\alpha}$$ and let $U_i=f(\mathscr{B}_{d_N^\Omega}(\widetilde{\mathbf{u}}_i,r_i))\cup \{\widetilde{\mathbf{v}}_i^*\}$.

According to the definition of $d_{N}^{\Omega}$,  we have 
\begin{align*}
\diam_{d_N^\Delta}(f(\mathscr{B}_{d_N^\Omega}(\widetilde{\mathbf{u}}_i,r_i)))
\leq  C{\diam_{d_N^\Omega}(\mathscr{B}_{d_N^\Omega}(\widetilde{\mathbf{u}}_i,r_i))}^{\alpha}=  C(2r_i)^{\alpha}
\leq  C\varepsilon^{\alpha}
\end{align*}
holds for every $N\in\mathbb{N}$.

For every $i\in\mathbb{N}$ and $N>N(\delta,\varepsilon)$, choose a close ball $\mathscr{B}_{d^\Delta_N}({\widetilde{\mathbf{v}}}_{i}^{c},l_i)$ with $l_i =\diam_{d_N^\Delta}(U_i)$ and $U_i \subseteq \mathscr{B}_{d^\Delta_N}({\widetilde{\mathbf{v}}}_{i}^{c},l_{i})$. So $\{\mathscr{B}_{d^\Delta_N}({\widetilde{\mathbf{v}}}_{i}^{c},l_{i})\}_{i\in\mathbb{N}}$ covers $f(\widetilde{\Xi})$, and then
\begin{align*}
    \Psi_{1}(4C\varepsilon^{\alpha}) \leq 2l_i &\leq 2d^{\Delta}_{N}(\widetilde{\mathbf{v}}_i,\widetilde{\mathbf{v}}_i^*)+2\diam_{d_N^\Delta}(f(\mathscr{B}_{d_N^\Omega}(\widetilde{\mathbf{u}}_i,r_i)))\\
    &\leq 2\frac{\Psi_{1}(4C\varepsilon^{\alpha})}{p} + 2C(2r_i)^{\alpha} \\&\leq 4C(2r_i)^{\alpha}
   \\& \leq 4C\varepsilon^{\alpha}.
\end{align*}

Therefore, we can deduce that
\begin{align*}
    \sum_{i\in\mathbb{N}}(2l_i)^{Nt}\leq (4C)^{Nt}\sum_{i\in\mathbb{N}}(2r_i)^{N\alpha t}\leq (4C)^{t}\sum_{i\in\mathbb{N}}(2r_i)^{Ns}\leq (4C)^{t}\delta.
\end{align*}
Thus, $t\geq \overline{\mdim}_{\Psi_{1}} (f(\widetilde{\Xi}),d^\Delta,\Gamma_\Delta)$ and $$\frac{1}{\alpha}\overline{\mdim}_{\Psi} (\widetilde{\Xi},d^\Delta,\Gamma_\Delta)\geq \overline{\mdim}_{\Psi_{1}} (f(\widetilde{\Xi}),d^\Delta,\Gamma_\Delta).$$
\end{proof}

\subsection{Product formulae} It is common to connect the dimensions of sets with those of their products. The following product formulas for the mean intermediate dimensions are not only intriguing in their regard but also practical for building examples.

Let $(\Omega,\Gamma_\Omega)$ and $(\Delta,\Gamma_\Delta)$ be two TDSs, and let $(\Omega,d^\Omega)$ and $(\Delta,d^\Delta)$ be two compact metric space. Let $\Omega\times \Delta$ denote the product space. 
For any $\widetilde{\mathbf{u}}=(\widetilde{\mathbf{u}}_1,\widetilde{\mathbf{u}}_2)\in \Omega\times \Delta$, define the continuous map $\Gamma_{\Omega\times \Delta}:\Omega\times \Delta \to \Omega\times \Delta$ as, 
$$
\Gamma_{\Omega\times \Delta}(\widetilde{\mathbf{u}})=\Big(\Gamma_{\Omega}(\widetilde{\mathbf{u}}_1),\Gamma_{\Delta}(\widetilde{\mathbf{u}}_2) \Big).
$$
Then we say $(\Omega\times \Delta,\Gamma_{\Omega\times \Delta})$ is a TDS which is introduced by $(\Omega,\Gamma_\Omega)$ and $(\Delta,\Gamma_\Delta)$.

\begin{proposition}\label{p351}
Let $(\Omega,\Gamma_\Omega)$ and $(\Delta,\Gamma_\Delta)$ be two TDSs and $\Psi$ be a admissible function. Consider the TDS $(\Omega\times \Delta,  \Gamma_{\Omega\times \Delta} )$ which is introduced by $(\Omega,\Gamma_\Omega)$ and $(\Delta,\Gamma_\Delta)$, 
if $d^{\Omega\times \Delta}$ is the metric on $\Omega\times \Delta$ such that for all $\widetilde{\mathbf{u}}=(\widetilde{\mathbf{u}}_1,\widetilde{\mathbf{u}}_2),\widetilde{\mathbf{v}}=(\widetilde{\mathbf{v}}_1,\widetilde{\mathbf{v}}_2)\in \Omega\times \Delta$, 
\begin{align*}
    d^{\Omega\times \Delta}(\widetilde{\mathbf{u}},\widetilde{\mathbf{v}})=\frac{1}{2} \max\left\{d^{\Omega}(\widetilde{\mathbf{u}}_1, \widetilde{\mathbf{v}}_1),d^{\Delta}(\widetilde{\mathbf{u}}_2,\widetilde{\mathbf{v}}_2) \right\},
\end{align*}
then
$$\overline{\mdim}_{ \Psi}(\Omega \times \Delta,d^{\Omega \times \Delta},\Gamma_{\Omega\times \Delta}) \leq \overline{\mdim}_{\Psi} (\Omega,d^\Omega,\Gamma_\Omega)+\overline{\mdim}^{M} (\Delta,d^\Delta,\Gamma_\Delta)$$
and
$$ \underline{\mdim}_{ \Psi}(\Omega \times \Delta,d^{\Omega \times \Delta},\Gamma_{\Omega\times \Delta}) \leq \underline{\mdim}_{\Psi} (\Omega,d^\Omega,\Gamma_\Omega)+\overline{\mdim}^{M} (\Delta,d^\Delta,\Gamma_\Delta).$$
\end{proposition}
\begin{proof}
First, we aim to prove the first inequality.
Obviously, for $N\in\mathbb{N}$, we have
\begin{align*}
     d^{\Omega\times \Delta}_{N}(\widetilde{\mathbf{u}},\widetilde{\mathbf{v}})=\frac{1}{2} \max\left\{d^{\Omega}_{N}(\widetilde{\mathbf{u}}_1, \widetilde{\mathbf{v}}_1),d^{\Delta}_{N}(\widetilde{\mathbf{u}}_2,\widetilde{\mathbf{v}}_2) \right\}.
\end{align*}
Fix $s>\overline{\mdim}_{\Psi}(\Omega,d^\Omega,\Gamma_\Omega)$ and $t>\overline{\mdim}^{M}(\Delta,d^\Delta,\Gamma_\Delta)$, so the following claims hold.

\begin{enumerate}
\item[(1)] There is $0<\varepsilon_M <1$, for all $0<\varepsilon<\varepsilon_M$, there is a $M(\varepsilon)\in\mathbb{N}$ which satisfies for all $N>M(\varepsilon)$ there is a cover $\{\mathscr{B}_{d_{N}^{\Delta}}(\widetilde{\mathbf{v}}_i,r_i)\}_{i}$ of $\Delta$ with $2r_i=\varepsilon$ for every $i$, 
$$
\sum_{i}(2r_i)^{Nt}\leq 1, 
$$
i.e., for all $0<\varepsilon<\varepsilon_M$ and $N>M(\varepsilon)$, there is a cover of $\Delta$ by $\varepsilon^{-Nt}$ or fewer sets, each having diameter at most $\varepsilon$,

\item[(2)] Let $0<\delta<1$, there is a $0<\varepsilon_{\Psi}<1$, for all $0<\varepsilon<\varepsilon_{\Psi}$ there is a $K(\delta,\varepsilon)\in\mathbb{N}$ such that there exists a cover $\{\mathscr{B}_{d_{N}^{\Omega}}(\widetilde{\mathbf{u}}_k,l_k)\}_{k}$ of $\Omega$ with $\Psi(\varepsilon) < 2l_k \leq \varepsilon$ for every $k$,
$$
\sum_{k} (2l_k)^{Ns}\leq \delta.
$$
\end{enumerate}
Let $\varepsilon< \min\{\varepsilon_M , \varepsilon_{\Psi}\}$ and $N> \max\{ M(\varepsilon) , K(\delta,\varepsilon)\}$. For the cover $\{\mathscr{B}_{d_{N}^{\Omega}}(\widetilde{\mathbf{u}}_k,l_k)\}_{k}$ of $\Omega$ in claim $(2)$, using $2l_k$ as the diameter of the cover of $\Delta$ in claim $(1)$ for every $k$ respectively, i.e., for each $k$ there is a cover $\{\mathscr{B}_{d_{N}^{\Delta}}(\widetilde{\mathbf{v}}_{k,j},r_{k,j})\}_{j}$ of $\Delta$ by $(2l_k)^{-Nt}$ or fewer sets and $r_{k,j}= l_k$ for every $j$. Obviously,
$$
\Omega\times \Delta\subseteq \bigcup_{k}\bigcup_{j} \mathscr{B}_{d_{N}^{\Omega}}(\widetilde{\mathbf{u}}_k,l_k)\times \mathscr{B}_{d_{N}^{\Delta}}(\widetilde{\mathbf{v}}_{k,j},r_{k,j}).
$$
Then for all $k$ and $j$,
\begin{align*}
 \Psi(\varepsilon)<  2l_k&= \max\{2r_{k,j},2l_k\}\\
&=2\diam_{d_N^{\Omega\times \Delta}}(\mathscr{B}_{d_{N}^{\Omega}}(\widetilde{\mathbf{u}}_k,l_k)\times \mathscr{B}_{d_{N}^{\Delta}}(\widetilde{\mathbf{v}}_{k,j},r_{k,j}))\\
&\leq\varepsilon,
\end{align*}
and choose a close ball $\mathscr{B}_{d_{N}^{\Omega\times \Delta}}(z_{k,j},p_{k,j})\subseteq \Omega\times \Delta$ with $$p_{k,j}={\diam_{d_N^{\Omega\times \Delta}}(\mathscr{B}_{d_{N}^{\Omega}}(\widetilde{\mathbf{u}}_k,l_k)\times \mathscr{B}_{d_{N}^{\Delta}}(\widetilde{\mathbf{v}}_{k,j},r_{k,j}))}$$
and $\mathscr{B}_{d_{N}^{\Omega}}(\widetilde{\mathbf{u}}_k,l_k)\times \mathscr{B}_{d_{N}^{\Delta}}(\widetilde{\mathbf{v}}_{k,j},r_{k,j})\subseteq \mathscr{B}_{d_{N}^{\Omega\times \Delta}}(z_{k,j},p_{k,j})$.
Therefore, $\Omega\times \Delta\subseteq \bigcup_{k}\bigcup_{j}\mathscr{B}_{d_{N}^{\Omega\times \Delta}}(z_{k,j},p_{k,j})$ and
\begin{align*}
\sum_{k}\sum_{j}(2p_{k,j})^{N(s+t)}\leq \sum_{k} (2l_k)^{-Nt} (2 l_k)^{N(s+t)} \leq \sum_{k}(2l_k)^{Ns}
\leq \delta.
\end{align*}
Then 
$$
\overline{\mdim}_{ \Psi}(\Omega\times \Delta,d^{\Omega\times \Delta},\Gamma_{\Omega\times \Delta})\leq s+t
$$ 
and 
$$
\overline{\mdim}_{ \Psi}(\Omega\times \Delta,d^{\Omega\times \Delta},\Gamma_{\Omega\times \Delta})\leq \overline{\mdim}_{\Psi} (\Omega,d^\Omega,\Gamma_\Omega)+\overline{\mdim}^{M} (\Delta,d^\Delta,\Gamma_\Delta).
$$

The proof of the second inequality is similar to that of the first inequality, only requiring an adjustment of the value of $\varepsilon$ to a suitable range.
\end{proof}

\begin{proposition}\label{p352}
Let $(\Omega,\Gamma_\Omega)$ and $(\Delta,\Gamma_\Delta)$ be two TDSs which have finite Assouad dimensions, $d^\Omega$ and $d^\Delta$ be two metrics such that $(\Omega,d^\Omega)$ and $(\Delta,d^\Delta)$ are compact, and $\Psi$ be an admissible function. 
Consider the TDS $(\Omega\times \Delta,  \Gamma_{\Omega\times \Delta} )$ which is introduced by $(\Omega,\Gamma_\Omega)$ and $(\Delta,\Gamma_\Delta)$. Let $d^{\Omega\times \Delta}$ be the metric on $\Omega\times \Delta$ such that for all $\widetilde{\mathbf{u}}=(\widetilde{\mathbf{u}}_1,\widetilde{\mathbf{u}}_2),\widetilde{\mathbf{v}}=(\widetilde{\mathbf{v}}_1,\widetilde{\mathbf{v}}_2)\in \Omega\times \Delta$, 
\begin{align*}
    d^{\Omega\times \Delta}(\widetilde{\mathbf{u}},\widetilde{\mathbf{v}})= \max\left\{d^{\Omega}(\widetilde{\mathbf{u}}_1, \widetilde{\mathbf{v}}_1),d^{\Delta}(\widetilde{\mathbf{u}}_2,\widetilde{\mathbf{v}}_2) \right\}.
\end{align*}
If $\underline{\mdim}_{\Psi}(\Omega,d^\Omega,\Gamma_\Omega)>0$ and $\underline{\mdim}_{\Psi}(\Delta,d^\Delta,\Gamma_\Delta)>0$, then 
$$
\underline{\mdim}_{\Psi}(\Omega,d^\Omega,\Gamma_\Omega)+\underline{\mdim}_{\Psi}(\Delta,d^\Delta,\Gamma_\Delta)\leq \underline{\mdim}_{\Psi}(\Omega \times \Delta,d^{\Omega\times \Delta},\Gamma_{\Omega\times \Delta} ).
$$
If $\overline{\mdim}_{\Psi}(\Omega,d^\Omega,\Gamma_\Omega)>0$ and $\underline{\mdim}_{\Psi}(\Delta,d^\Delta,\Gamma_\Delta)>0$, then 
$$
\overline{\mdim}_{\Psi}(\Omega,d^\Omega,\Gamma_\Omega)+\underline{\mdim}_{\Psi}(\Delta,d^\Delta,\Gamma_\Delta)\leq \overline{\mdim}_{\Psi}(\Omega \times \Delta,d^{\Omega\times \Delta},\Gamma_{\Omega\times \Delta} ).
$$

\end{proposition}
\begin{proof}
We focus solely on proving the second inequality since the proof of the first inequality follows a similar approach.

Let $s<\underline{\mdim}_{\Psi}(\Omega,d^\Omega,\Gamma_\Omega)$ and $t<\underline{\mdim}_{\Psi}(\Delta,d^\Delta,\Gamma_\Delta)$. By the assertion $(B)$ of Lemma \ref{l52}, there exists $c \in(0, \infty)$ such that for all sufficiently small $\varepsilon$ and sufficient large $N$ there exist two Borel probability measures $\nu_{\varepsilon,N}^{\Omega}$ and $\nu_{\varepsilon,N}^{\Delta}$ with finite supports such that if $\widetilde{\mathbf{u}} \in \Omega$, $\widetilde{\mathbf{v}}\in \Delta$ and $\Psi\left(\varepsilon\right) < r \leq \varepsilon$, we have
   $$
   \nu_{\varepsilon,N}^{\Omega}(\mathscr{B}_{d_N^\Omega} (\widetilde{\mathbf{u}}, r)) \leq c r^{Ns}
   $$
and
   $$
   \nu_{\varepsilon,N}^{\Delta}(\mathscr{B}_{d_N^\Delta} (\widetilde{\mathbf{v}}, r)) \leq c r^{Nt}.
   $$
Let 
$$
\vartheta_{\varepsilon,N}= \nu_{\varepsilon,N}^{\Omega} \times  \nu_{\varepsilon,N}^{\Delta}
$$ 
be the product measure on $\Omega\times \Delta$.

Let $U\subseteq \Omega\times \Delta$ be a Borel set with $\Psi(\varepsilon)<\diam_{d^{\Omega\times \Delta}_N}(U)\leq \varepsilon$. Then for $(\widetilde{\mathbf{u}},\widetilde{\mathbf{v}})\in U$,
\begin{align*}
    U\subseteq \mathscr{B}_{d^{\Omega\times \Delta}_{N}}\left((\widetilde{\mathbf{u}},\widetilde{\mathbf{v}}),2\diam_{d^{\Omega\times \Delta}_N}(U)\right)\subseteq \mathscr{B}_{d_N^\Omega}\left(\widetilde{\mathbf{u}},2\diam_{d^{\Omega\times \Delta}_N}(U)\right)\times
    \mathscr{B}_{d_N^\Delta}\left(\widetilde{\mathbf{v}},2\diam_{d^{\Omega\times \Delta}_N}(U)\right).
\end{align*}
It follows from the fact $\mdim^{A}(\Omega,d^\Omega,\Gamma_\Omega)$ and $\mdim^{A}(\Delta,d^\Delta,\Gamma_\Delta)$ are finite, that $\mathscr{B}_{d_N^\Omega}\left(\widetilde{\mathbf{v}},2\diam_{d^{\Omega\times \Delta}_N}(U)\right)$ and $\mathscr{B}_{d_N^\Delta}\left(\widetilde{\mathbf{u}},2\diam_{d^{\Omega\times \Delta}_N}(U)\right)$ can be covered by finite close balls $\left\{\mathscr{B}_{d_N^\Omega}(\widetilde{\mathbf{u}}_i,\diam_{d^{\Omega\times \Delta}_N}(U))\right\}_{i=1}^{M}$ and $\left\{\mathscr{B}_{d_N^\Delta}(\widetilde{\mathbf{v}}_j,\diam_{d^{\Omega\times \Delta}_N}(U))\right\}_{j=1}^{M}$ respectively, i.e. there exists $M>0$ such that
$$
\mathscr{B}_{d_N^\Omega}\left(\widetilde{\mathbf{u}},2\diam_{d^{\Omega\times \Delta}_N}(U)\right)\subseteq \bigcup_{i=1}^{M}\mathscr{B}_{d_N^\Omega}(\widetilde{\mathbf{u}}_i,\diam_{d^{\Omega\times \Delta}_N}(U))
$$
and
$$
\mathscr{B}_{d_N^\Delta}\left(\widetilde{\mathbf{v}},2\diam_{d^{\Omega\times \Delta}_N}(U)\right)\subseteq \bigcup_{i=1}^{M}\mathscr{B}_{d_N^\Delta}(\widetilde{\mathbf{v}}_j,\diam_{d^{\Omega\times \Delta}_N}(U)).
$$
Therefore, we have 
\begin{align*}
\mathscr{B}_{d_N^\Omega}\left(\widetilde{\mathbf{u}},2\diam_{d^{\Omega\times \Delta}_N}(U)\right)&\times \mathscr{B}_{d_N^\Delta}\left(\widetilde{\mathbf{v}},2\diam_{d^{\Omega\times \Delta}_N}(U)\right)\\
&\subseteq 
\bigcup_{i=1}^{M}\mathscr{B}_{d_N^\Omega}(\widetilde{\mathbf{u}}_i,\diam_{d^{\Omega\times \Delta}_N}(U)\times \bigcup_{j=1}^{M}\mathscr{B}_{d_N^\Delta}(\widetilde{\mathbf{v}}_j,\diam_{d^{\Omega\times \Delta}_N}(U))\\
&=
\bigcup_{i=1}^{M}\bigcup_{j=1}^{M}\Big( \mathscr{B}_{d_N^\Omega}(\widetilde{\mathbf{u}}_i,\diam_{d^{\Omega\times \Delta}_N}(U))\times \mathscr{B}_{d_N^\Delta}(\widetilde{\mathbf{v}}_j,\diam_{d^{\Omega\times \Delta}_N}(U)\Big),
\end{align*}
whcih implies that
\begin{align*}
\vartheta_{\varepsilon,N}(U) &\leq \vartheta_{\varepsilon,N}\left(\mathscr{B}_{d_N^\Omega}\left(\widetilde{\mathbf{u}},2\diam_{d^{\Omega\times \Delta}_N}(U)\right)\times \mathscr{B}_{d_N^\Delta}\left(\widetilde{\mathbf{v}},2\diam_{d^{\Omega\times \Delta}_N}(U)\right)\right)\\
&\leq \vartheta_{\varepsilon,N}\Big(\bigcup_{i=1}^{M}\bigcup_{j=1}^{M}( \mathscr{B}_{d_N^\Omega}(\widetilde{\mathbf{u}}_i,\diam_{d^{\Omega\times \Delta}_N}(U))\times \mathscr{B}_{d_N^\Delta}(\widetilde{\mathbf{v}}_j,\diam_{d^{\Omega\times \Delta}_N}(U)) \Big)\\
&\leq \sum_{i=1}^{M}\sum_{j=1}^{M}\vartheta_{\varepsilon,N}(\mathscr{B}_{d_N^\Omega}(\widetilde{\mathbf{u}}_i,\diam_{d^{\Omega\times \Delta}_N}(U))\times \mathscr{B}_{d_N^\Delta}(\widetilde{\mathbf{v}}_j,\diam_{d^{\Omega\times \Delta}_N}(U))\\
&\leq M^{2}c^{2}\diam_{(d\times d)_N}(U)^{N(s+t)}.
\end{align*}
Then by using $(B)$ of Lemma \ref{l321}, $t+s\leq \underline{\mdim}_{\Psi}(\Omega \times \Delta,d^{\Omega\times \Delta},\Gamma_{\Omega\times \Delta})$
and $$\underline{\mdim}_{\Psi}(\Omega,d^\Omega,\Gamma_\Omega)+\underline{\mdim}_{\Psi}(\Delta,d^\Delta,\Gamma_\Delta)\leq \underline{\mdim}_{\Psi}(\Omega \times \Delta,d^{\Omega\times \Delta},\Gamma_{\Omega\times \Delta}).$$
\end{proof}

Now let $n\in \mathbb{N}$ and $(\Omega,\Gamma)$ be a TDS. For all $\widetilde{\mathbf{u}}=(\widetilde{\mathbf{u}}_1,\widetilde{\mathbf{u}}_2,\cdots,\widetilde{\mathbf{u}}_n)\in \Omega^{n}$ and $\widetilde{\mathbf{v}}=(\widetilde{\mathbf{v}}_1,\widetilde{\mathbf{v}}_2,\cdots,\widetilde{\mathbf{v}}_n)\in \Omega^{n}$, the continues map $\Gamma^{n}:\Omega\times \Omega\to \Omega\times \Omega$ is defined by
$$
\Gamma^{n}(\widetilde{\mathbf{u}})=\Gamma^{n}(\widetilde{\mathbf{u}}_1,\widetilde{\mathbf{u}}_2,\cdots,\widetilde{\mathbf{u}}_n)=(\Gamma(\widetilde{\mathbf{u}}_1),\Gamma(\widetilde{\mathbf{u}}_2),\cdots,\Gamma(\widetilde{\mathbf{u}}_n)).
$$
Suppose $d$ be the metric on $\Omega$, let the metric $d^{n}$ satisfies 
\begin{align*}
    d^{n}(\widetilde{\mathbf{u}},\widetilde{\mathbf{v}})=\max\{d(\widetilde{\mathbf{u}}_1, \widetilde{\mathbf{v}}_1),d(\widetilde{\mathbf{u}}_2,\widetilde{\mathbf{v}}_2),\cdots,d(\widetilde{\mathbf{u}}_n,\widetilde{\mathbf{v}}_n) \},
\end{align*}
and for $N\in\mathbb{N}$,
\begin{align*}
    d^{n}_{N}(\widetilde{\mathbf{u}},\widetilde{\mathbf{v}})=\max\{d_{N}(\widetilde{\mathbf{u}}_1, \widetilde{\mathbf{v}}_1),d_{N}(\widetilde{\mathbf{u}}_2,\widetilde{\mathbf{v}}_2),\cdots,d_{N}(\widetilde{\mathbf{u}}_n,\widetilde{\mathbf{v}}_n) \}.
\end{align*}
Then we say $(\Omega^{n}, \Gamma^{n})$ is the $n$-dimensional TDS introduced by $(\Omega, \Gamma)$ itself.
Then as a direct consequence of Propositions \ref{p352}, we have the following corollary.
\begin{corollary}
Let $(\Omega,\Gamma)$ be a TDS where $(\Omega,d)$ is a compact metric space, and $\Psi$ be a admissible function. 
Consider the $n$-dimensional TDS $(\Omega^{n}, \Gamma^{n})$ with $n\geq 2$. If $\underline{\mdim}_{\Psi}(\Omega,d,\Gamma)>0$ and $\Omega$ have finite Assouad dimensions, then 
$$
\underline{\mdim}_{\Psi}(\Omega^{n},d^{n},\Gamma^{n})\geq n~\underline{\mdim}_{\Psi}(\Omega,d,\Gamma)
$$
\end{corollary}

\section{Some Examples}\label{Section5}
In this section, we will give three examples such that the upper mean $\Psi$-intermediate dimension is equal to the lower mean $\Psi$-intermediate dimension, and calculate its mean $\Psi$-intermediate dimension.
\begin{example}\cite[Example 2.3]{TM1}
Let $I=[0,1]$ and $I^{\mathbb{N}}$ denote the infinite-dimensional cube, and let $\Psi$ be an admissible function. We define the shift map $\sigma$ on $I^{\mathbb{N}}$ as follows  
\[  
\sigma\left((\widetilde{\mathbf{u}}_n)_{n\in \mathbb{N}}\right) = (\widetilde{\mathbf{u}}_{n+1})_{n\in \mathbb{N}}.  
\]  
Furthermore, we define a metric $d$ on $[0,1]^\mathbb{N}$ by  
\[  
d\left((\widetilde{\mathbf{u}}_n)_{n\in \mathbb{N}}, (\widetilde{\mathbf{v}}_n)_{n\in \mathbb{N}}\right) = \sum_{n=1}^\infty \frac{|\widetilde{\mathbf{u}}_n-\widetilde{\mathbf{v}}_n|}{2^n}.  
\]  
Then it holds that  
\[  
\mdim^{H}\left(I^{\mathbb{N}},d,\sigma\right) = \mdim^{M}\left(I^{\mathbb{N}},d,\sigma\right) = 1.  
\]  
Consequently, by using Lemma \ref{l3}, we obtain  
\[  
\mdim^{H}(I^{\mathbb{N}},d,\sigma) = \mdim_{\Psi}(I^{\mathbb{N}},d,\sigma) = \mdim^{M}(I^{\mathbb{N}},d,\sigma) = 1.  
\]

\end{example}

\begin{example}\cite[Proposition 4.1 and Theorem 4.3]{TM1}
Let $\Psi$ be an admissible function, and let  
\[  
\ell^\infty = \left\{(\widetilde{\mathbf{u}}_n)_{n\in \mathbb{N}}\in \mathbb{R}^{\mathbb{N}}\, \middle|\, \sup_{n\geq 1} |\widetilde{\mathbf{u}}_n| < \infty\right\}  
\]  
denote the space of bounded sequences, endowed with the norm $\Vert{\widetilde{\mathbf{u}}}\Vert_\infty := \sup_{n\geq 1} |\widetilde{\mathbf{u}}_n|$ for $\widetilde{\mathbf{u}} = (\widetilde{\mathbf{u}}_n)_{n\in \mathbb{N}}$. We always assume that $\ell^\infty$ is equipped with the weak$^*$ topology as the dual space of $\ell^1$.  
  
We define the shift map $\sigma:\ell^\infty \to \ell^\infty$ by  
\[  
\sigma\left((\widetilde{\mathbf{u}}_n)_{n\in \mathbb{N}}\right) = (\widetilde{\mathbf{u}}_{n+1})_{n\in \mathbb{N}},  
\]  
which is continuous with respect to the weak$^*$ topology.  
  
Consider a TDS $(K, \Gamma)$ and suppose that for each $\kappa\in K$, we are given a point $a(\kappa) = \left(a(\kappa)_n\right)_{n\in \mathbb{N}}\in \ell^\infty$ such that the map  
\[  
\kappa \mapsto a(\kappa)  
\]  
is continuous (with respect to the weak$^*$ topology of $\ell^\infty$) and equivariant (i.e., $\sigma\left(a(\kappa)\right) = a(\Gamma\kappa)$) for every $\kappa\in K$. Since $K$ is compact, it follows that  
\[  
\sup_{\kappa \in K} \Vert{a(\kappa)}\Vert_\infty < \infty.  
\]  
  
Fix a real number $0<c<1$. For each $\kappa \in K$, we define a transformation $S_\kappa:\ell^\infty\to \ell^\infty$ by  
\[  
S_\kappa (\widetilde{\mathbf{u}}) = c \widetilde{\mathbf{u}} + a(\kappa).  
\]  
This transformation satisfies $\sigma\left(S_\kappa(\widetilde{\mathbf{u}})\right) = S_{\Gamma\kappa}\left(\sigma(\widetilde{\mathbf{u}})\right)$.  
  
There exists a unique non-empty compact subset $\Omega$ of $\ell^\infty$ satisfying  
\[  
\Omega = \bigcup_{\kappa \in K} S_\kappa(\Omega).  
\]  
This set $\Omega$ is $\sigma$-invariant (i.e., $\sigma(\Omega) \subseteq \Omega$). The TDS $(\Omega, \sigma)$ is called a self-similar system defined by the family of contracting similarity transformations $\{S_\kappa\}_{\kappa\in K}$. Then the self-similar system $\Omega$ satisfies  
\[  
\mdim^{H}(\Omega,d,\sigma) = \mdim^{M}(\Omega, d,\sigma).  
\]  
Consequently, it follows from Lemma \ref{l3} that  
\[  
\mdim^{H}(\Omega,d,\sigma) = \mdim_{\Psi}(\Omega,d,\sigma) = \mdim^{M}(\Omega, d,\sigma).  
\]

\end{example}

\begin{example}
Let $0<\theta\leq 1$, $\Psi(\widetilde{\mathbf{u}})=\widetilde{\mathbf{u}}^{\frac{1}{\theta}}$ and $K = \{0\}\cup\left\{\frac{1}{n}\mid n\geq 1\right\} = \left\{0,1,\frac{1}{2},\frac{1}{3},\dots\right\}$, obviously, $\Psi$ is an admissible function.
Define the shift map $\sigma: K^{\mathbb{N}} \to K^{\mathbb{N}}$ by
$$
\sigma\left((\widetilde{\mathbf{u}}_n)_{n\in \mathbb{N}}\right) = (\widetilde{\mathbf{u}}_{n+1})_{n\in \mathbb{N}},
$$
where $(\widetilde{\mathbf{u}}_n)_{n\in \mathbb{N}}\in K^{\mathbb{N}}$. Define also a metric $d$ on $K^{\mathbb{N}}$ by
$$  d\left((\widetilde{\mathbf{u}}_n)_{n\in \mathbb{N}}, (\widetilde{\mathbf{v}}_n)_{n\in \mathbb{N}}\right) = \sum_{n=1}^\infty 2^{-n}|\widetilde{\mathbf{u}}_n-\widetilde{\mathbf{v}}_n|. $$
where $(\widetilde{\mathbf{u}}_n)_{n\in \mathbb{N}},(\widetilde{\mathbf{v}}_n)_{n\in \mathbb{N}}\in K^{\mathbb{N}}$. Then
\begin{align*} 
&\mdim^{H}\left(K^\mathbb{N},d,\sigma\right) = 0, \quad
\mdim^{M}\left(K^\mathbb{N},d,\sigma\right) = \frac{1}{2}
\quad\text{and}\quad
\underline{\mdim}^{\Psi}\left(K^\mathbb{N},d,\sigma\right) \geq\frac{\theta}{4}.
\end{align*}
The proof of $\mdim^{H}\left(K^\mathbb{N},d,\sigma\right) = 0$ is given by the Appendix of \cite{TM1}. The proof of $\mdim^{M}\left(K^\mathbb{N},d,\sigma\right) = \frac{1}{2}$ is mentioned in \cite[Example 2.4]{TM1}, and we will provide the specific calculation process. It should be emphasized that in the following calculations, the base of all logarithms is $2$.

For any $m\in\mathbb{N}$, the set
$$
\left\{\left[\frac{n-1}{4m^{2}},\frac{n+1}{4m^{2}}\right]\;\Big|\; 0\leq n\leq 2m^{2} \right\}\cup \left\{\left[1-\frac{1}{4m^2},1+\frac{1}{4m^2}\right]\right\}
$$
is a cover of $K$. Let $l=\lceil \log m^{2}\rceil+1$, it is easy to see that
\begin{equation}\label{exq1}
\sum_{i=l+1}^{\infty}\frac{1}{2^i}\leq \frac{1}{2m^{2}}
\end{equation}
and
\begin{equation}\label{exq2}
\sum_{i=1}^{l}\frac{1}{2^i}\frac{1}{2m^2}\leq \frac{1}{2m^2}.
\end{equation}

Thus, we have for any $N\in\mathbb{N}$
$$
\operatorname{span}(K^{\mathbb{N}},N,d,\sigma, \frac{1}{2m^2})\leq (3m)^{N+l},
$$
which implies that
\begin{align*}
\overline{\mdim}^{M}(K^{\mathbb{N}},d,\sigma)\leq \liminf_{m\to\infty}\limsup_{N\to\infty}\frac{N+l}{N}\frac{\log 3m}{\log 2m^{2}}=\frac{1}{2}.
\end{align*}

Let $0<\varepsilon<1$, then there exists $m\in\mathbb{N}$ such that 
$$
\frac{1}{m(m+1)}\leq \varepsilon< \frac{1}{m(m-1)}.
$$
Let $A_{\varepsilon}=\Big\{1,\frac{1}{2},\cdots,\frac{1}{m}\Big\}$ and $h=\lfloor -\log\varepsilon\rfloor$, then
\begin{equation}\label{exq3}
\sum_{i=h+1}^{\infty}\frac{1}{2^i}\geq \varepsilon.
\end{equation}
It is clear that
 $$
 \operatorname{sep}(K^{\mathbb{N}},N,d,\sigma, \varepsilon)\geq m^{N+h}
 $$
and
\begin{align*}
\underline{\mdim}^{M}(K^{\mathbb{N}},d,\sigma)\geq \limsup_{m\to\infty}\limsup_{N\to\infty}\frac{N+h}{N}\frac{\log m}{\log (m(m-1)/2)}=\frac{1}{2}.
\end{align*}
Finally, we calculate the lower bound of mean intermediate dimension with respect to the admissible function $\Psi(\widetilde{\mathbf{u}})=\widetilde{\mathbf{u}}^{\frac{1}{\theta}}$, where $0<\theta\leq 1$.

We consider a suitable measure on $K^{\mathbb{N}}$. Let $\varepsilon>0$ and $s>0$, then there exists $m\in\mathbb{N}$ with 
$$
\frac{1}{m(m+1)}< 2\varepsilon\leq \frac{1}{m(m-1)}.
$$

For $N\in\mathbb{N}$, set 
$$
\mathscr{M}_{N}=\Bigg\{ \prod_{i=1}^{N}\left\{\frac{1}{M_i}\right\}\times \prod_{i\geq N+1}K \;\Big|\; 1\leq i\leq N,M_i\in\mathbb{N}  \Bigg\}
$$ 
and 
$$
\mathscr{M}_{N}^{m}=\Bigg\{ \prod_{i=1}^{N}\left\{\frac{1}{M_i}\right\}\times \prod_{i\geq N+1}K \;\Big|\; 1\leq i\leq N, 1\leq M_i\leq m \Bigg\}.
$$ Clearly, $\mathscr{M}_{N}^{m}$ is a subset of $\mathscr{M}_{N}$ and $\mathscr{M}_{N}$ is a partition of $K^{\mathbb{N}}$ for every $N\in\mathbb{N}$ and let the set $A\in\mathscr{M}_{N}$ carry the mass $Q_{N}^{\varepsilon}(A)$ with
\begin{align*}
Q_{N}^{\varepsilon}(A)=
\begin{cases}
\varepsilon^\frac{N(1+\theta)s}{4\theta}, \quad &A\in\mathscr{M}_{N}^{m},\\\\
0, \quad         &A\notin\mathscr{M}_{N}^{m}.
\end{cases}
\end{align*}

For a Borel set $U$, we define the measure $\mu_{\varepsilon,N}$ as follows
\begin{align*}
\mu_{\varepsilon,N}(U)=\sum_{U\cap A\neq\emptyset,A\in\mathscr{M}_{N}} Q_{N}^{\varepsilon}(A).
\end{align*}
We always assume that $\varepsilon$ is sufficiently small, so we can always assume that $m>2$ holds without loss of generality, then 
\begin{align*}
\mu_{\varepsilon,N}\Big(\supp(\mu_{\varepsilon,N})\Big)&=m^{N}\varepsilon^\frac{N(1+\theta)s}{4\theta}\\
&\geq \frac{m^{N}}{(2m^{2}+2m)^{\frac{N(1+\theta)s}{4\theta}}}\\
&\geq \frac{m^{N}}{(4m^{2})^{\frac{N(1+\theta)s}{4\theta}}}\\&
=\frac{m^{N}}{2^{\frac{N(1+\theta)s}{2\theta}}m^{\frac{N(1+\theta)s}{2\theta}}}\\&\geq \frac{m^{N}}{m^{\frac{N(1+\theta)s}{\theta}}},
\end{align*}
if $s=\frac{\theta}{1+\theta}$, then $\mu_{\varepsilon,N}\Big(\supp(\mu_{\varepsilon,N})\Big)\geq 1$.

If $\varepsilon^{\frac{1}{\theta}}< \diam_{d_N}(U)\leq \varepsilon$, then 
$$
\mu_{\varepsilon,N}(U)\leq  \varepsilon^\frac{N(1+\theta)s}{4\theta}\leq \diam_{d_N}(U)^\frac{N(1+\theta)s}{4}.
$$
It follows from Lemma \ref{l321} that
\begin{align*}
\underline{\mdim}^{\Psi}(K^{\mathbb{N}},d,\sigma)\geq \frac{(1+\theta)s}{4}=\frac{\theta}{4}.
\end{align*}

\end{example}

\bigskip\bigskip\bigskip
\noindent{\bf Funding:} This work is supported by NSFC (No.12271432) and by Analysis, Probability \& Fractals Laboratory (No. LR18ES17).\\


\begin{thebibliography}{50}

\bibitem{JMA1} J.M. Acevedo., Genericity of continuous maps with positive metric mean dimension., Results Math. 77(2022), no.1, Paper No. 2, 30 pp.

\bibitem{JMA2} J.M. Acevedo., Genericity of homeomorphisms with full mean Hausdorff dimension., Regul. Chaotic Dyn. 29(2024), no.3, 474–490.

\bibitem{Backes} L. Backes and F.B. Rodrigues.,  A variational principle for the metric mean dimension of level sets.,  IEEE Trans. Inform. Theory. 69(2023), no.9, 5485–5496.

\bibitem{BA} A. Banaji., Generalised intermediate dimensions., Monatsh. Math. 202(2023), no.3, 465–506.

\bibitem{1} A. Banaji and H. Chen., Dimensions of popcorn-like pyramid sets., J. Fractal Geom. 10(2023), no.1-2, 151–168.

\bibitem{n1} A. Banaji and J.M. Fraser.,
Assouad type dimensions of infinitely generated self-conformal sets., Nonlinearity. 37(2024), no.4, Paper No. 045004, 32 pp.

\bibitem{2} A. Banaji and J.M. Fraser., Intermediate dimensions of infinitely generated attractors., Trans. Amer. Math. Soc. 376(2023), no.4, 2449–2479.

\bibitem{3} A. Banaji and I. Kolossváry., Intermediate dimensions of Bedford–McMullen carpets with applications to Lipschitz equivalence., Adv. Math. 449(2024), Paper No. 109735, 69 pp.

\bibitem{4} A. Banaji and A. Rutar., Attainable forms of intermediate dimensions., Ann. Fenn. Math. 47(2022), no.2, 939–960.

\bibitem{5} S.A. Burrell., Dimensions of fractional Brownian images., J. Theoret. Probab. 35(2022), no.4, 2217–2238.

\bibitem{6} S.A. Burrell, K.J. Falconer and J.M. Fraser., Projection theorems for intermediate dimensions., J. Fractal Geom. 8(2021), no.2, 95–116.

\bibitem{7} S.A. Burrell, K.J. Falconer and J.M. Fraser., The fractal structure of elliptical polynomial spirals., Monatsh. Math. 199(2022), no.1, 1–22.


\bibitem{Cheng} D. Cheng, Z. Li and B. Selmi., Upper metric mean dimensions with potential on subsets., Nonlinearity. 34(2021), no.2, 852–867.

\bibitem{9} L. Daw and G. Kerchev., Fractal dimensions of the Rosenblatt process., Stochastic Process. Appl. 161(2023), 544–571.

\bibitem{Douzi} Z. Douzi and B. Selmi., Projection theorems for Hewitt–Stromberg and modified intermediate dimensions., Results Math. 77(2022), no.4, Paper No. 158, 14 pp.

\bibitem{11} K.J. Falconer., Intermediate dimension of images of sequences under fractional Brownian motion., Statist. Probab. Lett. 182(2022), Paper No. 109300, 6 pp.

\bibitem{12} K.J. Falconer., Intermediate dimensions: a survey., In: Pollicott, M., Vaienti, S. (Eds.) Thermodynamic formalism. Springer Lecture Notes in Mathematics. 2290(2021), 469-494.

\bibitem{n4} K.J. Falconer and J.M. Fraser.,
Assouad dimension influences the box and packing dimensions of orthogonal projections., 
J. Fractal Geom. 8(2021), no.3, 247–259.

\bibitem{13} K.J. Falconer, J.M. Fraser and T. Kempton., Intermediate dimensions., Math. Z. 296(2020), 813-830.

\bibitem{15} J.M. Fraser., Assouad dimension and fractal geometry., Cambridge Tracts in Math., 222
Cambridge University Press, Cambridge, 2021. xvi+269 pp.

\bibitem{16} J.M. Fraser., Fractal geometry of Bedford-McMullen carpets., Thermodynamic formalism, 495–516.
Lecture Notes in Math., 2290
CIRM Jean-Morlet Ser.
Springer, Cham, (2021).


\bibitem{17} J.M. Fraser., Interpolating between dimensions., Fractal geometry and stochastics VI, 3–24.
Progr. Probab., 76
Birkhäuser/Springer, Cham, (2021).

\bibitem{n2} Z. Feng.,
Intermediate dimensions under self-affine codings., Math. Z. 307(2024), no.1, Paper No. 21, 29 pp.

\bibitem{19} J.M. Fraser, K.E. Hare, K.G. Hare, S. Troscheit and H. Yu., The Assouad spectrum and the quasi-Assouad dimension: a tale of two spectra., Ann. Fenn. Math. 44(2019), 379-387.

\bibitem{20} J.M. Fraser and H. Yu., New dimension spectra: finer information on scaling and homogeneity., Adv. Math. 329(2018), 273-328.


\bibitem{GS21} Y. Gutman and A. Spiewak., Around the variational principle for metric mean dimension., Studia Math. 261(2021), no.3, 345–360.

\bibitem{gro} M. Gromov., Topological invariants of dynamical systems and spaces of holomorphic maps., I. Math. Phys. Anal. Geom. 2(1999), 323-415.


\bibitem{22} I. Garcia, K.E. Hare and F. Mendivil., Almost sure Assouad-like dimensions of complementary sets., Math. Z. 298(2021), 1201–1220.


\bibitem{23} I. Garcia, K.E. Hare and F. Mendivil., Intermediate Assouad-like dimensions. J. Fractal Geom. 8(2021), no.3, 201–245.

\bibitem{24} K.E. Hare and F. Mendivil., Assouad-like dimensions of a class of random Moran measures., J. Math. Anal. Appl. 508(2022), no.2, Paper No. 125912, 25 pp.

\bibitem{n3} I. Kolossváry.,
An upper bound for the intermediate dimensions of Bedford-McMullen carpets., J. Fractal Geom. 9(2022), no.1-2, 151–169.

\bibitem{25} K.E. Hare and F. Mendivil., Assouad-like dimensions of a class of random Moran measures II. Nonhomogeneous Moran sets., J. Fractal Geom. 10(2023), no.3-4, 351–388.

\bibitem{28} T. Hytonen and A. Kairema., Systems of dyadic cubes in a doubling metric space., Colloq. Math. 126(2012), no.1, 1–33.

\bibitem{LY} L. Jin and Y. Qiao.,  Mean dimension of product spaces: a fundamental formula., Math. Ann. 388(2024), 249–259. 



\bibitem{L} E. Lindenstrauss., Mean dimension, small entropy factors and an embedding theorem., Inst. Hautes Études Sci. Publ. Math. (1999), no.89, 227–262.

\bibitem{LT} E. Lindenstrauss and M. Tsukamoto., From rate-distortion theory to metric mean dimension: variational principle., IEEE Trans. Inform. Theory. 64(2018), NO.5, 3590-3609.

\bibitem{LHF1} H. Li., Sofic mean dimension., Adv. Math. 244(2013), 570–604.

\bibitem{LHF2} H. Li and B. Liang., Mean dimension, mean rank, and von Neumann–Lück rank., J. Reine Angew. Math. 739(2018), 207–240.

\bibitem{Liu} Y. Liu, B. Selmi, and Z. Li.,  On the mean fractal dimensions of the cartesian product sets., Chaos Solitons Fractals. 180(2024), Paper No. 114503, 9 pp.

\bibitem{TM1} M. Tsukamoto., Mean Hausdorff dimension of Some infinite dimensional fractals., https://arxiv.org/abs/2209.00512.

\bibitem{34} J.C. Robinson and N. Sharples., Strict inequality in the box-counting dimension product formulas., Real Anal. Exchange. 38(2012/13), no.1, 95–119.


\bibitem{Shi22} R. Shi., On variational principles for metric mean dimension., IEEE Trans. Inform. Theory. 68(2022), no.7, 4282–4288.

\bibitem{36} J.T. Tan., On the intermediate dimensions of concentric spheres and related sets., https://arxiv.org/abs/2008.10564.

\bibitem{37} S. Troscheit., Assouad spectrum thresholds for some random constructions., Canad. Math. Bull. 63(2020), no.2, 434–453.

\bibitem{VV17} A. Velozo and R. Velozo., Rate distortion theory, metric mean dimension and measure theoretic entropy., https://arxiv.org/abs/1707.05762.




\bibitem{Wal82} P. Walter., An introduction to ergodic theory., Springer-Verlag, New York, (1982).

\bibitem{Wu21} W. Wu., On relative metric mean dimension with potential and variational principles., J. Dynam. Differential Equations.35(2023), no.3, 2313–2335.

\bibitem{Yang} R. Yang, E. Chen and X. Zhou., Bowen's equations for upper metric mean dimension with potential., Nonlinearity.35(2022), no.9, 4905–4938.

\end{thebibliography}
\end{document}